\documentclass[12pt]{amsart}
\usepackage{a4wide,enumerate,color,mathrsfs}
\allowdisplaybreaks
 
\let\pa\partial  
\let\na\nabla  
\let\eps\varepsilon  
\newcommand{\N}{{\mathbb N}}  
\newcommand{\R}{{\mathbb R}}
\newcommand{\Z}{{\mathbb Z}} 
\newcommand{\T}{{\mathbb T}}
\newcommand{\C}{{\mathbb C}}
\renewcommand{\S}{{\mathbb S}}
\renewcommand{\H}{{\mathcal H}}
\renewcommand{\ker}{{\mathcal N}}
\newcommand{\ran}{{\mathcal R}}
\newcommand{\D}{{\mathcal D}}
\newcommand{\e}{{\mathrm e}}
\newcommand{\I}{{\mathrm i}}

\newcommand{\blue}{\textcolor{black}}

\newtheorem{theorem}{Theorem}   
\newtheorem{lemma}[theorem]{Lemma}   
   
\newtheorem{remark}[theorem]{Remark}

\begin{document}  

\title[Hypocoercivity for a Boltzmann system]{Hypocoercivity for a linearized
multi-species Boltzmann system}

\author[E. S. Daus]{Esther S. Daus}
\address{Institute for Analysis and Scientific Computing, Vienna University of  
Technology, Wiedner Hauptstra\ss e 8--10, 1040 Wien, Austria}
\email{esther.daus@tuwien.ac.at}

\author[A. J\"{u}ngel]{Ansgar J\"{u}ngel}
\address{Institute for Analysis and Scientific Computing, Vienna University of  
Technology, Wiedner Hauptstra\ss e 8--10, 1040 Wien, Austria}
\email{juengel@tuwien.ac.at}

\author[C. Mouhot]{Cl\'{e}ment Mouhot}
\address{DPMMS, University of Cambridge\\Wilberforce Road,
Cambridge CB3 0WA, United Kingdom}
\email{c.mouhot@dpmms.cam.ac.uk}

\author[N. Zamponi]{Nicola Zamponi}
\address{Institute for Analysis and Scientific Computing, Vienna University of  
Technology, Wiedner Hauptstra\ss e 8--10, 1040 Wien, Austria}
\email{nicola.zamponi@tuwien.ac.at}

\date{\today}

\thanks{The authors thank Christian Schmeiser for very fruitful discussions and 
valuable help and the referees for their careful reading and helpful corrections. 
Furthermore, they acknowledge partial support from   
the Austrian Science Fund (FWF), grants P24304, P27352, and W1245, and    
the Austrian-French Program of the Austrian Exchange Service (\"OAD)}  

\begin{abstract}
  A new coercivity estimate on the spectral gap of the linearized
  Boltzmann collision operator for multiple species is proved. The
  assumptions on the collision kernels include hard and Maxwellian
  potentials under Grad's angular cut-off condition. Two proofs are
  given: a non-constructive one, based on the decomposition of the
  collision operator into a compact and a coercive part, and a
  constructive one, which exploits the ``cross-effects'' coming from
  collisions between different species and which yields explicit
  constants. Furthermore, the essential spectra of the linearized
  collision operator and the linearized Boltzmann operator are
  calculated. Based on the spectral-gap estimate, the exponential
  convergence towards global equilibrium with explicit rate is shown
  for solutions to the linearized multi-species Boltzmann system on
  the torus. The convergence is achieved by the interplay between the
  dissipative collision operator and the conservative transport
  operator and is proved by using the hypocoercivity method of Mouhot
  and Neumann.
\end{abstract}



\maketitle


\section{Introduction}\label{sec.intro}

This paper is concerned with the proof of (explicit) spectral-gap estimates
of the linearized Boltzmann operator for gas mixtures in the case of hard and
Maxwellian potentials as well as
the exponential decay of solutions to a {\em multi-species} Boltzmann system.
Spectral-gap estimates and the large-time behavior of the {\em mono-species} 
Boltzmann equations were intensively studied in the literature, 
but are unknown for {\em multi-species} systems.
First, we review the literature for the mono-species case.

The study of the linearized collision operator,
in the spatially homogeneous and hard-potential case, goes back to 
Hilbert \cite{Hil12}. For this operator, Carleman \cite{Car57} proved the 
existence of a spectral gap. The results were extended by Grad \cite{Gra58} for 
hard potentials with cut-off. Baranger and Mouhot \cite{BaMo05}
derived constructive estimates in the hard-sphere case. For Maxwell
molecules, Fourier transform methods were employed in \cite{WUB70}
to achieve explicit spectral properties. 
A spectral-gap estimate for the linearized Boltzmann operator, 
consisting of the sum of the 
linearized collision operator and the transport operator, was first shown by
Ukai \cite{Uka74}. Improved estimates (in smaller spaces of Sobolev type), 
still for hard potentials, 
were established in \cite{MoNe06}. In \cite{MoSt07}, spectral-gap estimates for 
moderately soft potentials (without angular cut-off) were proved, improving and 
extending previous results by Pao \cite{Pao74}. Hypoelliptic estimates for
the linearized operator without cut-off can be found in \cite{AHL12} and
references therein. A spectral analysis with relaxed 
tail decay and regularity conditions on the solutions was performed 
recently in an abstract framework \cite{GMM10}. Dolbeault et al.\ \cite{DMS15}
derived exponential decay rates in weighted $L^2$ spaces, which improves previous
Sobolev estimates. For further references, we refer to \cite[Section 1.5]{MoSt07}.

Spectral properties of the linearized Boltzmann operator were already investigated
by Grad \cite{Gra63}. Based on these results, Schechter \cite{Sch66} located the
essential spectrum of the classical collision operator in $L^2$. The spectrum
of the Boltzmann operator for hard spheres was also analyzed in $L^p$ for 
$p\neq 2$; see \cite{Kla75}. We refer to the recent work \cite{Dud13} for 
further results in $L^p$ for $1\le p\le\infty$ and more references. 
A detailed analysis of the resolvent and spectrum of the linearized Boltzmann
operator can be found in \cite[Section~2.2]{UkYa06}. A complete analysis for
the essential and discrete spectra for the linearized 
collision operator with hard potentials was performed in \cite{Mou06a}.

All these results are valid for the linearized {\em mono-species} collision operator.
Our aim is to extend the spectral-gap analysis to the case of the linearized 
{\em multi-species}
Boltzmann system modeling an ideal gas mixture. This is achieved by generalizing
the coercivity method of \cite{MoNe06}, including
quantitative estimates on the spectral gap for the multi-species collision operator.
A crucial step of our analysis is the observation that the multi-species version
of the $H$-theorem implies conservation of mass for each species but conservation
of momentum and energy only for the sum of all species. 
As a consequence, we need to study carefully the ``cross-effects'' of the collisions, 
i.e., how collisions between different species act on distribution functions which 
are elements of the nullspace of the mono-species collision operator. The crucial 
step is to relate these ``cross-effects'' to the differences of momentum 
and energy. Before stating the main results, we introduce the kinetic setting.


\subsection{The Boltzmann equation}\label{sec.BE}

The evolution of a dilute ideal gas composed of $n\ge 2$ different species of
chemically non-interacting mono-atomic particles (see \cite{DMS05} for chemically
reacting gases) with the same particle mass can be modeled by the following 
system of Boltzmann equations, stated on the three-dimensional torus $\T^3$,
\begin{equation}\label{BE}
  \pa_t F_i + v\cdot\na_x F_i = Q_i(F), \quad t>0,
	\quad F_i(x,v,0)=F_{I,i}(x,v),\quad (x,v)\in\T^3\times\R^3,
\end{equation}
where $1\le i\le n$. The vector $F=(F_1,\ldots,F_n)$ is the distribution function
of the system, with $F_i$ describing the $i$th species. The variables
are the position $x\in\T^3$, the velocity $v\in\R^3$, and the time $t\ge 0$.
The right-hand side of the kinetic equation in \eqref{BE} is the
$i$th component of the nonlinear collision operator, defined by
$$
  Q_i(F) = \sum_{j=1}^n Q_{ij}(F_i,F_j), \quad 1\le i\le n,
$$
where $Q_{ij}$ models interactions between particles of the same ($i=j$)
or of different species ($i\neq j$),
$$
  Q_{ij}(F_i,F_j)(v) = \int_{\R^3\times\S^2}B_{ij}(|v-v^*|,\cos\vartheta)
	(F_i'F_j'^* - F_iF_j^*)dv^*d\sigma,
$$
with the abbreviations $F_i'=F_i(v')$, $F_i^*=F_i(v^*)$, $F_i'^*=F_i(v'^*)$,
the three-dimensional unit sphere $\S^2$, and 
\begin{equation}\label{vprime}
  v' = \frac{v+v^*}{2} + \frac{|v-v^*|}{2}\sigma, \quad
	v'^* = \frac{v+v^*}{2} - \frac{|v-v^*|}{2}\sigma
\end{equation}
are the pre-collisional velocities depending on the post-collisional velocities
$(v,v^*)$. These expressions follow from the fact that we assume the collisions to be
elastic, i.e., the momentum and kinetic energy are conserved on the microscopic
level:
\begin{equation}\label{cons}
  v' + v'^* = v + v^*, \quad \frac12|v'|^2 + \frac12|v'^*|^2 
	= \frac12|v|^2 + \frac12|v^*|^2.
\end{equation}
The collision kernels $B_{ij}$ are nonnegative functions of the modulus $|v-v^*|$
and the cosine of the deviation angle $\vartheta\in[0,\pi]$, defined by
$\cos\vartheta=\sigma\cdot(v-v^*)/|v-v^*|$. 

Although we will analyze a linearized version of $Q_{i}$, let us
recall the main properties of the nonlinear operator $Q_i$.  Using the
techniques from \cite[pp.~36-42]{CIP94}, it is not difficult to see
that $Q:=(Q_1,\dots,Q_n)$ conserves the mass of each species but
only total momentum and energy, i.e.
$$
  \int_{\R^3}\sum_{i,j=1}^n Q_{ij}(F_i,F_j)\psi_i(v)dv = 0
$$
if and only if $\psi(v)\in\mbox{span}\{e^{(1)},\ldots,e^{(n)},v_1{\mathbf 1},
v_2{\mathbf 1},v_3{\mathbf 1},|v|^2{\mathbf 1}\}$, where $e^{(i)}$ is the $i$th
unit vector in $\R^n$ and ${\mathbf 1}=(1,\ldots,1)\in\R^n$. It is shown in
\cite{DMS05} that $Q$ satisfies a multi-species version of the $H$-theorem
which implies that any local equilibrium, i.e.\ any function $F$ being the
maximum of the Boltzmann entropy, has the form of a local Maxwellian
$M_{\rm loc}=(M_{{\rm loc},1},\ldots,M_{{\rm loc},n})$ with
$$
  F_i(x,v,t) = M_{{\rm loc},i}(x,v,t) 
	= \frac{\rho_{{\rm loc},i}(x,t)}{(2\pi\theta_{\rm loc}(x,t))^{3/2}}
	\exp\left(-\frac{|v-u_{\rm loc}(x,t)|^2}{2\theta_{\rm loc}(x,t)}\right),
$$
where, introducing the total local density 
$\rho_{\rm loc}=\sum_{i=1}^n\rho_{{\rm loc},i}$,
$$
  \rho_{{\rm loc},i} = \int_{\R^3}F_idv, \quad
	u_{\rm loc} = \frac{1}{\rho_{\rm loc}}\sum_{i=1}^n\int_{\R^3}F_i vdv, \quad
	\theta_{\rm loc} = \frac{1}{3\rho_{\rm loc}}\sum_{i=1}^n\int_{\R^3}F_i|v-u|^2 dv
$$
are the (local) masses of the species, the total momentum and total
energy, respectively.

On the other hand, the global equilibrium, which is the unique
stationary solution $F$ to \eqref{BE}, is given by
$M=(M_1,\ldots,M_n)$ with
$$ 
  F_i(x,v) = M_i(v) = \frac{\rho_{\infty,i}}{(2\pi\theta_\infty)^{3/2}}
  \exp\left(-\frac{|v-u_\infty|^2}{2\theta_\infty}\right),
$$
where now, setting $\rho_\infty=\sum_{i=1}^n\rho_{\infty,i}$,
$$
  \rho_{\infty,i} = \int_{\T^3\times\R^3}F_{i}dxdv, \quad
	u_\infty = \frac{1}{\rho_\infty}\int_{\T^3\times\R^3}F_{i}vdxdv, \quad
	\theta_\infty = \frac{1}{3\rho_\infty}\int_{\T^3\times\R^3}F_{i}|v-u|^2dxdv
$$
do not depend on $(x,t)$. By translating and scaling the coordinate system,
we may assume that $u_\infty=0$ and $\theta_\infty=1$ 
such that the global equilibrium becomes
\begin{equation}\label{M}
  M_i(v) = \frac{\rho_{\infty,i}}{(2\pi)^{3/2}}e^{-|v|^2/2}, \quad 1\le i\le n.
\end{equation}


\subsection{Linearized Boltzmann collision operator}\label{sec.lin}

We assume that the distribution function $F_i$ is close to the global
equilibrium such that we can write $F_i = M_i + M_i^{1/2}f_i$ for some
small perturbation $f_i$, where $M_i$ is given by \eqref{M}. Then,
dropping the small nonlinear remaining term, $f_i$ satisfies
the linearized equation
\begin{equation}\label{LBE}
  \pa_t f_i + v\cdot\na_x f_i = L_i(f), \quad t>0,\quad
	f_i(x,v,0)=f_{I,i}(x,v), \quad (x,v)\in\T^3\times\R^3,
\end{equation}
for $1\le i\le n$, where $f=(f_1,\ldots,f_n)$ and
the $i$th component of the linearized collision operator 
$L=(L_1,\ldots,L_n)$ is given by
$$
  L_i(f) = \sum_{j=1}^n L_{ij}(f_i,f_j), \quad 1\le i\le n,
$$
with
\begin{align}
  L_{ij}(f_i,f_j) &= M_i^{-1/2}\big(Q_{ij}(M_i,M_j^{1/2}f_j)
	+ Q_{ij}(M_i^{1/2}f_i,f_j)\big) \nonumber \\
	&= \int_{\R^3\times\S^2}B_{ij}M_i^{1/2}M_j^*(h_i'+h_j'^*-h_i-h_j^*)dv^*d\sigma,
	\quad h_i := M_i^{-1/2}f_i. \label{Lij}
\end{align}
Here, we have used $M_i^{'*} M_j' = M_i^* M_j$ for any
$i$, $j$, which follows from \eqref{cons}. Notice that we have
chosen the linearization considered in, e.g.,
\cite{MoNe06,UkYa06}. Another linearization is given by
$F_i=M_i+M_ig_i$ (see, e.g., \cite{Mou06}), namely
$\widetilde
L_{ij}(g_i,g_j)=M_i^{-1}(Q_{ij}(M_i,M_ig_i)+Q_{ij}(M_ig_i,M_i))$.
\blue{This choice gives the same results as with the linearization \eqref{Lij}
since both linearizations correspond to the same space of solutions,
but it turned out that the computations are easier using \eqref{Lij}.}

The linearized Boltzmann system satisfies an $H$-theorem
with the linearized entropy $H(f) = \frac12\sum_{i=1}^n\int_{\R^3}f_i^2dv$, 
$$
  -\frac{dH}{dt} = -\sum_{i=1}^n\int_{\R^3}f_i L_i(f)dv =: -(f,L(f))_{L^2_v} \ge 0,
$$
where $(\cdot,\cdot)_{L_v^2}$ is the scalar product on
$L_v^2:=L^2(\R^3;\R^n)$.  We will prove in Lemma \ref{lem.H} that
$(f,L(f))_{L^2_v}=0$ if and only if $M_i^{-1/2}f_i$ lies in
$\mbox{span}\{e^{(1)},\ldots,e^{(n)},v_1{\mathbf 1}, v_2{\mathbf 1},
v_3{\mathbf 1},|v|^2{\mathbf 1}\}$,
which is the null space $\ker(L)$ of the linear operator $L$.  The
main aim of this paper is to show that, under suitable assumptions on
the collision kernels, there exists a constant $\lambda>0$, which can
be computed explicitly, such that for all suitable functions $f$,
$-(f,L(f))_{L_v^2} \ge \lambda\|f-\varPi^L(f)\|_\H^2$, where
$\varPi^L$ is the projection onto $\ker(L)$ and $\H$ is a subset of
$L_v^2$ (see Theorem \ref{thm.spec} for the precise statement). This
spectral-gap estimate, together with hypocoercivity techniques, allows
us to conclude that exponential decay of the solutions $f(t)$
towards the global equilibrium holds (see Theorem \ref{thm.conv}).


\subsection{Assumptions on the collision kernels}\label{sec.assump}

We impose the following assumptions on the collision kernels $B_{ij}$
arising in \eqref{Lij}. 

\renewcommand{\labelenumi}{(A\theenumi)}
\begin{enumerate}
\item The collision kernels satisfy 
$$
  B_{ij}(|v-v^*|,\cos\vartheta) = B_{ji}(|v-v^*|,\cos\vartheta)
	\quad\mbox{for }1\le i,j\le n.
$$
\item The collision kernels decompose in the kinetic part $\Phi_{ij}\ge 0$ 
and the angular part $b_{ij}\ge 0$ according to
$$
  B_{ij}(|v-v^*|,\cos\vartheta) = \Phi_{ij}(|v-v^*|)b_{ij}(\cos\vartheta),
	\quad 1\le i,j\le n.
$$
\item For the kinetic part, there exist constants $C_1$, $C_2>0$,
$\gamma\in[0,1]$, and $\delta\in(0,1)$ such that for all $1\le i,j\le n$ and $r>0$,
$$
  C_1 r^\gamma\le \Phi_{ij}(r) \le C_2(r + r^{-\delta}).
$$
\item For the angular part, there exist constants $C_3$, $C_4>0$ such that
for all $1\le i,j\le n$ and $\vartheta\in[0,\pi]$,
$$
  0<b_{ij}(\cos\vartheta)\le C_3|\sin\vartheta|\,|\cos\vartheta|, \quad
	b'_{ij}(\cos\vartheta)\le C_4.
$$
Furthermore, 
$$
  C^b := \min_{1\le i\le n}\inf_{\sigma_1,\sigma_2\in\S^2}\int_{\S^2}\min\big\{
	b_{ii}(\sigma_1\cdot\sigma_3),b_{ii}(\sigma_2\cdot\sigma_3)\big\}d\sigma_3 > 0.
$$
\item For all $1\le i,j\le n$, 
$b_{ij}$ is even in $[-1,1]$ and the mapping $v\mapsto\Phi_{ij}'(|v|)$ on
$\R^3$ is locally integrable on $\R^3$ and bounded as $|v|\to\infty$.
\item There exists $\beta>0$ such that for all $1\le i,j\le n$, $s>0$, 
and $\sigma\in[-1,1]$, we have $B_{ij}(s,\sigma)\le \beta B_{ii}(s,\sigma)$.
\end{enumerate}

Following \cite{MoNe06}, since the functions $b_{ij}$ are integrable, we define
\begin{equation}\label{ellb}
  \ell^b := \min_{1,\le i,j\le n}\int_0^\pi b_{ij}(\cos\theta)\sin\theta d\theta > 0.
\end{equation}

Let us discuss these assumptions. The first hypothesis (A1) means that
the collisions are micro-reversible. Assumption (A2) is satisfied, for
instance, for collision kernels derived from interaction potentials
behaving like inverse-power laws. The lower bound in hypothesis (A3)
includes power-law functions $\Phi_{ij}(r)=r^\gamma$ with $\gamma >0$
(hard potential) and $\gamma=0$ (Maxwellian molecules). The assumption
$\gamma\ge 0$ is crucial since the linearized collision operator in
the mono-species case for soft potentials ($\gamma<0$) with angular
cut-off has no spectral gap \cite{BaMo05}; however, degenerate
spectral-gap estimates are possible \cite{GoPo86,Mou04}. 
The upper bound in (A3) means that the kinetic part is of
restricted growth for both small and large values of $|v-v^*|$.  In
hypothesis (A4), the upper bound for $b_{ij}$  implies Grad's cut-off
assumption. The positivity of $C^b$ in Assumption (A4) 
is used in the constructive
proof of the multi-species spectral-gap estimate (Theorem \ref{thm.spec})
via the {\em mono-species} spectral-gap estimate which depends on $C^b$;
see also the proofs of Theorem 1.1 in \cite{BaMo05} and
Theorem 6.1 in \cite{Mou04}. The positivity of $C^b$ 
is satisfied for the main physical case of a collision kernel satisfying Grad's
cut-off, i.e.\ for hard spheres with $B_{ij}(|v-v^*|,\cos \vartheta) = |v-v^*|$. 
Conditions (A1)-(A4) are also imposed in \cite{BaMo05,Mou04,Mou06} 
for the linearized mono-species Boltzmann operator.

  Assumption (A5) imposes technical conditions needed to verify the
  abstract hypotheses in \cite{MoNe06}. More precisely, the evenness
  of $b_{ij}$ is employed to show hypothesis (H2) (see section
  \ref{sec.conv}) and the properties on $\Phi_{ij}'$ are used to
  verify \eqref{H1.2} in hypothesis (H1).  The conditions on
  $\Phi_{ij}'$ are satisfied for hard and Maxwellian power-law
  potentials $\Phi_{ij}(r) = r^\gamma$ with exponent
  $\gamma\in [0,1]$, for instance.  Finally, condition (A6) states
  that the ratio of the off-diagonal and diagonal collision kernels
  can be bounded uniformly from above by a constant $\beta>0$. This
  hypothesis will be needed for the explicit computation of the
  constants in Theorems \ref{thm.spec} and \ref{thm.conv}.  More
  precisely, (A6) allows us to estimate the
  mono-species part of the collision operator using the computation of
  \cite{Mou04}; see Lemma \ref{lem.specLm}.


\subsection{Notation and definitions}\label{sec.not}

\blue{We call $\mbox{Dom}(F)$ the domain of an operator $F$ and 
$\mbox{Im}(f)$ the image of a mapping $f$.}
We introduce the spaces $L_v^2=L^2(\R^3;\R^n)$, $L^2_{x,v}=L^2(\T^3\times\R^3;\R^n)$,
$H^1_{x,v}=H^1(\T^3\times\R^3;\R^n)$, and
\begin{equation}\label{H}
  \begin{aligned}
  \H &= \bigg\{f\in L_v^2:\|f\|_\H^2 = \sum_{i=1}^n\int_{\R^3}f_i^2\nu_i dv 
	< \infty\bigg\}, \\
  \D &= \bigg\{f\in L_v^2:\|f\|_\H^2 =
  \sum_{i=1}^n\int_{\R^3}f_i^2\nu_i ^2 dv
	< \infty\bigg\}.
  \end{aligned}
\end{equation}
Here, $\nu_i$ is the collision frequency, given by 
\begin{equation}\label{nu}
  \nu_i(v) = \sum_{j=1}^n\int_{\R^3\times\S^2}B_{ij}(|v^*-v|,\cos\vartheta)M_j^*
	dv^*d\sigma, \quad i=1,\ldots,n,
\end{equation}
\blue{For Maxwellian modelcules $\Phi_{ij}(r)=\mbox{const.}$, the collision
frequency is constant but for strictly hard potentials $\Phi_{ij}(r)=r^\gamma$ 
with $0<\gamma\le 1$, $\nu_i$ is unbounded. In fact, it satisfies
$\nu_0(1+|v|)^\gamma\le \nu_i(v)\le \nu_1(1+|v|)^\gamma$ for some
constants $\nu_1\ge \nu_0>0$ \cite[p.~991]{MoNe06}. 
In the physically most relevant case of hard spheres ($\gamma=1$, $b_{ij}=1$),
the collision frequency can be computed explicitly, see formula (2.13) in
\cite[Section 7.2]{CIP94}. For more properties of
the collision frequencies, we refer to \cite[Section III.3]{Cer69}.}
If the collision frequencies are bounded, $\H=L_v^2$. Generally,
$\nu_i$ is unbounded and so, $\H$ is a proper subset of $L_v^2$.
The norm on $L_v^2$ (and similarly for the other spaces) is defined by
$$
  \|f\|_{L_v^2}^2 = \sum_{i=1}^n\int_{\R^3}f_i^2dv \quad\mbox{for }
	f=(f_1,\ldots,f_n)\in L_v^2.
$$

We distinguish the following linear operators.  We define the operator
$\Lambda=(\Lambda_1,\ldots,\Lambda_n):\mbox{Dom}(\Lambda)\to L_v^2$
by 
$$
  \Lambda_i(f)=\nu_i f_i, \quad i=1,\ldots,n,
$$
where $\mbox{Dom}(\Lambda)=\{f\in L_{v}^2:\Lambda f\in L_{v}^2\}=\D$.  It
is closed, densely defined, selfadjoint and, by Lemma \ref{lem.nu}
below, coercive.  The linearized collision operator
$L:\mbox{Dom}(L)\to L_v^2$, introduced in section \ref{sec.lin},
can be written as $L=K-\Lambda$, where $K:=L+\Lambda$, \blue{or, more explicitly,
$$
  K_i(f) = \sum_{j=1}^n\int_{\R^3\times\S^2}B_{ij}M_i^{1/2}M_j^*(h_i'+h_j'^*-h_j^*)
	dv^*d\sigma, \quad i=1,\ldots,n.
$$} 
It was shown in \cite{BGPS13}
that $K$ is a compact operator in $L_v^2$.  Thus,
$\mbox{Dom}(L)=\mbox{Dom}(\Lambda)=\D$ and $L$ is closed and
densely defined.  Furthermore, $L$ is nonpositive and selfadjoint on
$L_v^2$. We define the transport operator
$$
  T=v\cdot\na_x:\mbox{Dom}(T)\to L_v^2, 
$$
where $\mbox{Dom}(T)=\{f\in L_{x,v}^2:v\cdot\na_x f\in L_{x,v}^2\}$.
Finally, we consider the linearized Boltzmann operator
$$
  B=L-T:\mbox{Dom}(B)\to L_v^2,
$$
which is unbounded, closed, and densely defined
with $\mbox{Dom}(B)=\mbox{Dom}(L)\cap\mbox{Dom}(T)$.

We denote by $\ker(A)$ and $\ran(A)$ the kernel and range of a linear operator
$A$, respectively. Its resolvent set is denoted by $\rho(A)$ and its spectrum
by $\sigma(A)=\C\backslash\rho(A)$. For a linear unbounded operator $A$
with $\sigma(A)\subset(-\infty,0]$, we say that $A$ has a spectral gap when the
distance between 0 and $\sigma(A)\backslash\{0\}$ is positive.
Finally, the essential spectrum of $A$ is defined as the set of all 
complex numbers $\lambda\in\C$ such that $A-\lambda I$ is not Fredholm, where
$I$ is the identity operator. We refer to section \ref{sec.spec1} for
details regarding this definition.


\subsection{Main results}\label{sec.main}

In this subsection, we state the main results of the paper.  The first
result is a geometric property of the essential spectrum of the
linearized collision operator $L$ and the linearized Boltzmann
operator $B=L-T$.

\begin{theorem}[Essential spectrum of $L$ and $L-T$]\label{thm.ess}
Let the collision kernels $B_{ij}$ satisfy assumptions (A1)-(A4) and set 
$J=\cup_{i=1}^n \mathrm{Im}(\nu_i)\subset[\nu_0,\infty)$, where
$\nu_0=\min_{i=1,\ldots,n}\sup_{v\in\R^3}\nu_i(v)$ $>0$ (see Lemma \ref{lem.nu}).
Then
$$
  \sigma_{\rm ess}(L)=-J, \quad
	\sigma_{\rm ess}(L-T) = \{\lambda\in\C:\Re(\lambda)\in -J\}.
$$
\end{theorem}

\begin{remark}\rm We observe that if
$\lim_{|v|\to\infty}\nu_i(v)=\infty$ for $i=1,\ldots,n$ then
$$
  \sigma_{\rm ess}(L) = (-\infty,-\nu_0], \quad
	\sigma_{\rm ess}(L-T) = \{\lambda\in\C:\Re(\lambda)\le-\nu_0\}.
$$
Indeed, under the assumption $\nu_i(v)\to\infty$ as $|v|\to\infty$, the continuity
of $\nu_i$, and the Weierstra{\ss} theorem show that $J=[\nu_0,\infty)$.
Thus, the essential spectrum of the linearized {\em multi-species} collision
operator is very similar to the {\em mono-species} operator, 
where $\nu_0$ corresponds to the infimum in $\R^3$ of the single collision frequency;
see \cite[Section 3]{Mou06} and \cite[Prop.~3.1]{Mou06a}.
\qed
\end{remark}

The proof of Theorem \ref{thm.ess} is based on perturbation theory 
\cite[Chap.~IV]{Kat80} and is similar to the proof for the mono-species
collision operator \cite{UkYa06}, \blue{but we show new explicit spectral-gap 
estimates related to the particular structure of the kernel in the multi-species
case.}
More precisely, we write $L=K-\Lambda$ as described in
section \ref{sec.not}. It turns out that $K=L+\Lambda$ is compact on
$L_v^2$ (see section \ref{sec.spec1} for details). Weyl's theorem
\cite[Theorem S]{GuWe69} states that the essential spectrum of $L=K-\Lambda$
coincides with that of $-\Lambda$. Thus it remains to show that
$\sigma_{\rm ess}(\Lambda)=J$. This is done by using Weyl's singular
sequences, which allow for a sufficient and necessary condition for 
$\lambda\in\C$ being an element of the essential spectrum of the selfadjoint 
operator $\Lambda$.

The proof of the second statement in Theorem \ref{thm.ess}
is more involved since $K$ is not compact on $L_{x,v}^2$ and hence, Weyl's theorem
cannot be applied directly. The idea
is to employ an extended Weyl theorem, which states that the essential spectrum is
conserved under a {\em relatively compact} perturbation 
\cite[Section IV.5.6, Theorem 5.35]{Kat80}. Indeed, if $K$ is relatively 
compact with respect to $\Lambda+T$ then
$\sigma_{\rm ess}(L-T)=\sigma_{\rm ess}(K-(\Lambda+T))
=-\sigma_{\rm ess}(\Lambda+T)$, and it remains to compute the essential
spectrum of $\Lambda+T$.
\medskip

The next theorem concerns an explicit spectral-gap estimate. It is the main
result of the paper.

\begin{theorem}[Explicit spectral-gap estimate]\label{thm.spec}
Let the collision kernels $B_{ij}$ satisfy assumptions (A1)-(A4).
Then there exists a constant $\lambda>0$ such that
\begin{equation}\label{spec.gap}
  -(f,L(f))_{L_v^2} \ge \lambda\|f-\varPi^L(f)\|_\H^2
	\quad\mbox{for all }f\in\D,
\end{equation}
where $\varPi^L$ is the projection onto the null space $\ker(L)$.  If
additionally hypothesis (A6) holds, the constant $\lambda$ 
can be computed explicitly:
$$
  \lambda = \frac{\eta D^b}{8C_k}, \quad
	\eta = \min\left\{1,\frac{4C^mC_k}{16C_k+D^b}\right\},
$$
where $C^m$, $D^b$, and $C_k$ are defined in \eqref{Cm}, \eqref{D}, and \eqref{Ck}, 
respectively.
\end{theorem}

Note that the constant $C^m$ depends on the mono-species spectral-gap
constant $C^b$ via \eqref{Cm} below.
We present two proofs of this theorem. The first proof is non-constructive
and relies on an abstract functional theoretical argument, based on the
decomposition $L=K-\Lambda$ and Weyl's perturbation theorem. This abstract
spectral-gap estimate is proved in Lemma \ref{lem.spec.gap}.
The second proof provides a constructive spectral-gap estimate, generalizing
the result in \cite{Mou06} (also see \cite[Theorem 6.1]{Mou04}) from the
mono-species to the multi-species case. For this, we split the operator $L=L^m+L^b$
in the mono-species part $L^m=(L^m_1,\ldots,L^m_n)$ and the bi-species part 
$L^b=(L^b_1,\ldots,L^b_n)$, 
$$
  L^m_i(f_i) = L_{ii}(f_i,f_i), \quad L_{i}^b(f) = \sum_{j\neq i}L_{ij}(f_i,f_j).
$$
The proof consists of four main steps.

{\em Step 1: Coercivity of the mono-species operator $L^m$.}
The bi-species part of $L$ satisfies $-(f,L^b(f))_{L_v^2}\ge 0$
for all $f\in\D$. Furthermore,
the results of \cite[Theorem 6.1]{Mou04} show that for the mono-species part,
\begin{equation}\label{coerc.Lm}
  -(f,L^m(f))_{L_v^2} \ge C^m\|f-\varPi^{m}(f)\|_\H^2 \quad\mbox{for }f\in
	\mbox{Dom}(L^m),
\end{equation}
where the constant $C^m>0$ can be computed explicitly
and $\varPi^{m}$ is the projection onto $\ker(L^m)$ 
(see Lemma \ref{lem.specLm}). 
Inequality \eqref{coerc.Lm} may be interpreted as a coercivity estimate 
for $L^m$ on $\ker(L^m)^\perp$. It is related to the ``microscopic coercivity'' 
in \cite[Section 1.3]{DMS15} for the mono-species setting.
Hence, we obtain the ``naive'' spectral-gap estimate 
$$
  -(f,L(f))_{L_v^2}\ge C^m\|f-\varPi^{m}(f)\|_\H^2 \quad\mbox{for }f\in\D.
$$
This estimate is not sharp enough for the multi-species case since we need
an inequality for all $f-\varPi^L(f)\in\ker(L)^\perp$ and not only for 
$f-\varPi^m(f)\in\ker(L^m)^\perp
\subset \ker(L)^\perp$. By projecting onto $\ker(L^m)^\perp$ only, we neglect
the ``cross-effects'' coming from the bi-species part of the collision operator.
Thus, we need a better estimate for $-(f,L(f))_{L_v^2}$, which is achieved 
as follows.

{\em Step 2: Absorption of the orthogonal parts.}
The contribution $f^\perp:=f-\varPi^m(f)$ in $-(f,L^b(f))_{L_v^2}$ can be absorbed 
by the $\H$ norm of $f^\perp$ (see Lemma \ref{lem.ortho}), giving for a certain
$\eta>0$,
$$
  -(f,L(f))_{L_v^2} \ge (C^m-4\eta)\|f^\perp\|_\H^2 
	- \frac{\eta}{2}(\varPi^m(f),L^b(\varPi^m(f)))_{L_v^2}.
$$

{\em Step 3: Coercivity of the bi-species operator $L^b$.}
The projection $\varPi^m(f)$ depends on the velocities $u_i$ and energies $e_i$ of
the $i$th species, and thus, the cross terms can be bounded by the differences
of momentum and differences of energies,
\begin{equation}\label{coerc.Lb}
  -(\varPi^m(f),L^b(\varPi^m(f)))_{L_v^2} \ge C\sum_{i,j=1}^n\big(|u_i-u_j|^2
	+ |e_i-e_j|^2\big),
\end{equation}
for some constant $C>0$. This is the key step of the proof.
The inequality may be considered as a coercivity estimate for
the bi-species operator. A key observation is that the differences of
momenta and energies converge to zero as $f$ approaches the global equilibrium.
Whereas \eqref{coerc.Lm} acts on $\ker(L^m)^\perp$, \eqref{coerc.Lb}
gives an estimate on the orthogonal complement $\ker(L^m)$.

{\em Step 4: Lower bound for the differences of momenta and energy.}
The last step consists in estimating the
differences $|u_i-u_j|$ and $|e_i-e_j|$ from below by the error made by
projecting onto $\ker(L^m)^\perp$ instead of $\ker(L)^\perp$:
$$
  \sum_{i,j=1}^n\big(|u_i-u_j|^2 + |e_i-e_j|^2\big)
	\ge C\|f-\varPi^L(f)\|_\H^2 - 2C\|f-\varPi^m(f)\|_\H^2.
$$
Putting together the above inequalities, Theorem \ref{thm.spec}
follows; we refer to section \ref{sec.spec2} for details. 

As a consequence of the spectral-gap estimate, we are able to prove the
exponential decay of the solution $f(t)$ to \eqref{LBE} to the global
equilibrium with an explicit decay rate. 

\begin{theorem}[Convergence to equilibrium]\label{thm.conv}
Let the collision kernels $B_{ij}$ satisfy assumptions (A1)-(A5) and let
$f_I\in H_{x,v}^1$. Then the linearized Boltzmann operator $B=L-T$ generates
a strongly continuous semigroup $\e^{tB}$ on $H^1_{x,v}$, 
which satisfies
\begin{equation}\label{decay1}
  \|\e^{tB}(I-\varPi^B)\|_{H_{x,v}^1} \le C\e^{-\tau t}, \quad t\ge 0,
\end{equation}
for some constants $C$, $\tau>0$. In particular, the solution 
$f(t)=\e^{tB}f_I$ to \eqref{LBE} satisfies
\begin{equation}\label{decay2}
  \|f(t)-f_\infty\|_{H_{x,v}^1} \le C\e^{-\tau t}\|f_I-f_\infty\|_{H_{x,v}^1},
	\quad t\ge 0,
\end{equation}
where $f_\infty:=\varPi^B(f_I)$ is the global equilibrium of
\eqref{LBE}.  Moreover, under the additional assumption (A6) and
lower bound in (A4), the constants $C$ and $\tau$ depend only on
the constants appearing in hypotheses (H1)-(H3) in section
\ref{sec.conv} and in particular on $\lambda$ defined in Theorem
\ref{thm.spec}.
\end{theorem}

The idea of the proof is to employ the hypocoercivity of the
linearized Boltzmann operator $L-T$, using the interplay between the
degenerate-dissipative properties of $L$ and the conservative properties of
$T$. The aim is to find a functional $G[f]$ which is equivalent to the square
of the norm of a Banach space (here, $H_{x,v}^1$),
$$
  \kappa_1\|f\|_{H_{x,v}^1}^2 \le G[f] \le 
	\kappa_2\|f\|_{H_{x,v}^1}^2 \quad\mbox{for }
	f\in H_{x,v}^1,
$$
leading to
\begin{equation}\label{dGdt}
  \frac{d}{dt}G[f(t)] \le -\kappa\|f(t)\|_{H_{x,v}^1}^2, \quad t>0,
\end{equation}
where $\kappa_1$, $\kappa_2$, $\kappa>0$ and $f(t)=\e^{tB}f_I$. 
These two estimates yield exponential convergence
of $f(t)$ in $H_{x,v}^1$. It turns out that the obvious choice 
$G[f]=c_1\|f\|_{L_{x,v}^2}+c_2\|\na_x f\|_{L_{x,v}^2}+c_3\|\na_v f\|_{L_{x,v}^2}$
does not lead to a closed estimate. The key idea, inspired from \cite{Vil06}
and worked out in \cite{MoNe06}, is to add the ``mixed term'' 
$c_4(\na_x f,\na_v f)_{L_{x,v}^2}$ to the definition of $G[f]$. Then
$$
  \frac{d}{dt}(\na_x f,\na_v f)_{L_{x,v}^2} = -\|\na_x f\|_{L_{x,v}^2}^2
	+ 2(\na_x L(f),\na_v f)_{L_{x,v}^2},
$$
and the last term can be estimated in terms of expressions arising
from the time derivative of the other norms in $G[f(t)]$.  Thus,
choosing $c_i>0$ in a suitable way, one may conclude that \eqref{dGdt}
holds.

In \cite{MoNe06}, the calculation of \eqref{dGdt} is reduced to the
validity of certain abstract conditions on the operators $K$ and
$\Lambda$ (see section \ref{sec.conv}). These conditions state that
$\Lambda$ is coercive in a certain sense, $K$ has a regularizing
effect, and $L=K-\Lambda$ has a local spectral gap.  The last
condition is proved in Theorem \ref{thm.spec}, while the other
conditions follow from direct calculations, since the operators $K$ and
$\Lambda$ are given explicitly. As a consequence, the proof of
  Theorem \ref{thm.conv} essentially consists in verifying the
  abstract conditions stated in \cite{MoNe06}.  In contrast to the
  estimate of Theorem \ref{thm.spec}, where the multi-species
  character plays a role in the spectral-gap estimate, there are no
  ``cross-effects'' here and the same modified functional $G[f]$ as
  above, including the mixed term, can be used. However, the decay
  rate $\tau$ changes, since the constant in hypothesis (H3) (see
  section \ref{sec.conv}) differs in the mono- and multi-species case
  and $\tau$ depends also on that constant.

\blue{
We finish the introduction by commenting possible generalizations.
First, the convergence result based on hypocoercivity requires some regularity on the
initial data, namely $f_I\in H_{x,v}^1$. The extension of the exponential decay
to initial data from $L_{x,v}^2$ might be done by using the method of
\cite{GMM10}, which is based on a high-order factorization argument on the 
resolvents and semigroups. The proof of exponential decay is expected to 
be constructive and to preserve the optimal rate.}

\blue{Second, it seems to be not trivial to extend the results to the whole-space 
case. The problem is that one loses the compactness in the $x$-space. One possibility
is to assume some confinement potential which, under some appropriate weighted
Poincar\'e inequality, can yield compactness of the resolvent and hence a spectral
gap. For instance, Duan \cite{Dua11} used non-constructive techniques to 
prove decay rates for the mono-species linearized Boltzmann equation.
Still in the mono-species case, with one-dimensional collisional invariants and using
constructive methods, the decay is investigated by, e.g., H\'erau and Nier 
\cite{HeNi04} and Villani \cite{Vil06}, working in the space $H^1_{x,v}$, 
and by H\'erau \cite{Her06} and Dolbeault, Mouhot, Schmeiser \cite{DMS15},
working in the space $L_{x,v}^2$. The tasks in the multi-species case
are first to extend the non-constructive methods, which probably does not 
contain new difficulties, and second to
devise a constructive method, which is more involved and work in progress;
see \cite{DHMS15}.}

%

\blue{
Third, a Cauchy theory for the full nonlinear multi-species Boltzmann equation 
in a perturbative regime is work in progress \cite{BrDa15}.}

The paper is organized as follows. In section \ref{sec.prop}, some properties
of the linearized collision operator \eqref{Lij} are collected. Theorem 
\ref{thm.ess} on the essential spectrum of $L$ and $L-T$ and the abstract
spectral-gap estimate in Theorem \ref{thm.spec} are proved in section \ref{sec.spec1}.
We present a second proof of Theorem \ref{thm.spec}
in section \ref{sec.spec2}, by exploiting the
conservation properties and leading to explicit constants. 
Finally, Theorem \ref{thm.conv} is shown in section \ref{sec.conv}.


\section{Properties of the kinetic model}\label{sec.prop}

We show some properties of the linearized collision operator \eqref{Lij}
and the collision frequencies \eqref{nu}. Let assumptions (A1)-(A4) hold.
First we prove an $H$-theorem for \eqref{Lij}.

\begin{lemma}[$H$-theorem for the linearized collision operator]\label{lem.H}
It holds $(f,L(f))_{L_v^2}\le 0$ for all $f\in\D$ and $(f,L(f))_{L_v^2}=0$
if and only if $f\in\ker(L)$, where
\begin{align*}
  \ker(L) &= \big\{f\in L_v^2: \exists \alpha_1,\ldots,\alpha_n,\, e\in\R,\ u\in\R^3,\
	\forall 1\le i\le n, \\
	&\phantom{xxx} f_i = M_i^{1/2}(\alpha_i + u\cdot v+e|v|^2)\big\},
\end{align*}
and $M_i$ is given by \eqref{M}.
\end{lemma}

The proof is similar to the mono-species case except that the 
elements of the null space of $L$ depend on the total mean velocity $u$
and total energy $e$ instead of the individual velocities and energies.
Therefore, we give a complete proof. We note that an $H$-theorem for
the nonlinear Boltzmann operator for a mixture of reactive gases was proved in
\cite{DMS05}.

\begin{proof}
By the change of variables $(v,v^*)\mapsto (v^*,v)$ and $(v,v^*)\mapsto(v',v'^*)$
and the symmetry of $B_{ij}$ (assumption (A1)), we can write for $f\in L_v^2$,
$$
  (f,L(f))_{L_v^2} = -\frac14\sum_{i,j=1}^n\int_{\R^6\times\S^2}B_{ij}M_iM_j^*
	(h_i'+h_j'^*-h_i-h_j^*)^2 dv^*dvd\sigma,
$$
where we recall that $h_i=M_i^{-1/2}f_i$. This shows that $(f,L(f))_{L_v^2}\le 0$
for all $f\in \D$. Moreover, $(f,L(f))_{L_v^2}=0$ if and only if 
\begin{equation}\label{NL.1}
  h_i'+h_j'^*-h_i-h_j^*=0\quad\mbox{for all }(v,v^*)\in\R^3\times\R^3,\ 1\le i,j\le n. 
\end{equation}	
It is shown in \cite[pp.~36-42]{CIP94} that \eqref{NL.1} for $i=j$ implies that
$h_i$ has the form $h_i(v)=\alpha_i + u_i\cdot v + e_i|v|^2$ for suitable
constants $\alpha_i$, $e_i\in\R$ and $u_i\in\R^3$. Inserting this expression into
\eqref{NL.1} leads to
\begin{equation}\label{NL.2}
  u_i\cdot(v'-v) + u_j\cdot(v'^* - v^*) + e_i(|v'|^2 - |v|^2) 
	+ e_j(|v'^*|^2 - |v^*|^2) = 0
\end{equation}
for $1\le i,j\le n$.
We consider the particular type of collisions with $v'=v^*$, $v'^*=v$, and
$|v|=|v'|$. For such collisions, $\sigma=(v^*-v)/|v^*-v|$. Then the above
equation becomes
$$
  (u_i-u_j)\cdot(v'-v) = 0, \quad 1\le i,j\le n.
$$
By rotating the velocities $v$, $v'$ in all possible ways, we deduce that
$(u_i-u_j)\cdot w=0$ for all $w\in\R^3$ and thus, $u_i=u_j$ for all $1\le i,j\le n$.
We set $u:=u_1$. This fact, together with the conservation of momentum 
$v'-v+v'^*-v^*=0$, implies that \eqref{NL.2} becomes
$$
  e_i(|v'|^2-|v|^2) + e_j(|v'^*|^2-|v^*|^2) = 0, \quad 1\le i,j\le n.
$$
Taking into account the conservation of energy $|v'|^2-|v|^2+|v'^*|^2-|v^*|^2=0$,
we infer that $(e_i-e_j)(|v'|^2-|v|^2)=0$ and consequently, $e_i=e_j$ for all
$1\le i,j\le n$. Set $e:=e_1$. We have shown that $(f,L(f))_{L_v^2}=0$ if and only if
there exist $\alpha_1,\ldots,\alpha_n$, $e\in\R$ and $u\in\R^3$ such that
for $1\le i\le n$, $f_i(v)=M_i^{1/2}(\alpha_i+u\cdot v+e|v|^2)$. These functions
clearly belong to $\ker(L)$, which finishes the proof.
\end{proof}

The next result is concerned with the stationary solutions of \eqref{LBE}.

\begin{lemma}\label{lem.glob.eq}
The global equilibrium $f_\infty=(f_{\infty,1},\ldots,f_{\infty,n})$
of \eqref{LBE}, i.e.\ the unique stationary solution, is given by
$$
  f_{\infty,i}(v) = M_i^{1/2}(\alpha_i + u\cdot v + e|v|^2),\quad 1\le i\le n,
$$
where $\alpha_i$, $e\in\R$ and $u\in\R^3$ are uniquely determined by the
global conservation laws of mass, momentum, and energy, i.e.\ by the
equations
$$
  \int_{\R^3}M_i^{1/2}(f_{\infty,i}-f_{I,i})\psi(v)dv = 0, 
	\quad 1\le i\le n, 
$$
for $\psi(v)=1,v_1,v_2,v_3,|v|^2$, where $f_{I,i}$ are the initial data.
\end{lemma}

\begin{proof}
First, we claim that $\ker(B)=\ker(L)\cap\ker(T)$, where $B=L-T$
and $T=v\cdot\na_x$ are considered on $\T^3\times\R^3$.
The inclusion $\ker(L)\cap\ker(T)\subset\ker(B)$ being trivial, let $f\in\ker(B)$.
Then, using the skew-symmetry of $T$,
$$
  0 = (f,B(f))_{L_{x,v}^2} = (f,L(f))_{L_{x,v}^2} - (f,T(f))_{L_{x,v}^2} 
	= (f,L(f))_{L_{x,v}^2}.
$$
Lemma \ref{lem.H} shows that $f\in\ker(L)$. But this implies that
$T(f)=L(f)-B(f)=0$ and hence $f\in\ker(T)$. This shows the claim.
Let $f_\infty$ be a stationary solution. Then $f_\infty\in\ker(B)$
and by our claim, $f\in\ker(L)\cap\ker(T)$.
Since $\ker(T)=\{f\in L_{x,v}^2:\na_x f=0\}$ \cite[Lemma B.2]{Bri13}, 
$f_\infty$ does not depend on $x$. Because of $f_\infty\in\ker(L)$,
Lemma \ref{lem.H} shows the result.
\end{proof}

Finally, we prove that the collision frequencies \eqref{nu} are strictly positive
with bounded derivative.

\begin{lemma}\label{lem.nu}
Let Assumptions (A2)-(A4) hold.
The collision frequencies \eqref{nu} satisfy
\begin{equation}\label{Cnu}
  \min_{1\le i\le n}\inf_{v\in\R^3} \nu_i(v) 
	\ge \nu_0 := 2^{3\gamma/2}\frac{C_1\ell^b\rho_\infty}{\sqrt{\pi}}
	\Gamma\left(\frac{\gamma+3}{2}\right) > 0,
\end{equation}
where $C_1>0$ is given by assumption (A3), 
$\ell^b>0$ is defined in \eqref{ellb},
$\rho_\infty:=\sum_{j=1}^n \rho_{j,\infty}$ (see \eqref{M}), 
and $\Gamma$ is the Gamma function. Furthermore,
if additionally (A5) holds, then $\na_v\nu_i\in L_v^\infty(\R^3)$,
implying that $|\nu_i(v)|\le C_\nu(1+|v|)$ for some $C_\nu>0$ and for all
$v\in\R^3$, $i=1,\ldots,n$.
\end{lemma}

\begin{proof}
\blue{The decomposition of $B_{ij}$, according to assumption (A2), implies that
$$
  \nu_i(v) = \sum_{j=1}^n\int_{\R^3} \Phi_{ij}(|v-v^*|)M_j^* dv^*
	\int_{\S^2}b_{ij}(\cos\vartheta)d\sigma.
$$
The integral
$$
  c_{ij} := \int_{\S^2}b_{ij}(\cos\vartheta)d\sigma = 2\pi\int_0^\pi 
	b_{ij}(\cos\vartheta)\sin\vartheta d\vartheta
$$
does not depend on $v$ or $v^*$. We conclude from (A2)-(A4) that
\begin{align}\label{nu.pos}
  \nu_i(v) 
	&= (2\pi)^{-3/2}\sum_{j=1}^nc_{ij}\rho_{\infty,j}
	\int_{\R^3}\Phi_{ij}(|v-v^*|)\e^{-|v^*|^2/2}dv^* \\
	&\ge \frac{C_1\ell^b\rho_\infty}{(2\pi)^{3/2}}\int_{\R^3}|v-v^*|^\gamma 
	\e^{-|v^*|^2/2}dv^*. \nonumber
\end{align}
}
Observe that the function
$$
  G(v) := \int_{\R^3}|v-v^*|^\gamma \e^{-|v^*|^2/2}dv^*
$$
is \blue{uniformly} positive since
the transformation $v^*\mapsto -v^*$ and the elementary inequality
$$
  |v-v^*|^\gamma + |v+v^*|^\gamma \ge |(v-v^*)+(v+v^*)|^\gamma = 2^\gamma|v^*|^\gamma
$$
for $\gamma\in[0,1]$ lead to
$$
  G(v) = \frac12\int_{\R^3}(|v-v^*|^\gamma + |v+v^*|^\gamma)\e^{-|v^*|^2/2}dv^*
	\ge 2^{\gamma-1}\int_{\R^3}|v^*|^\gamma \e^{-|v^*|^2/2}dv^* = 2^{\gamma-1}G(0).
$$
Actually, using spherical coordinates and the change of unknowns $s=r^2/2$,
$$
  G(0) = 4\pi\int_0^\infty r^{\gamma+2} \e^{-r^2/2}dr 
	= 2^{(\gamma+5)/2}\pi\int_0^\infty s^{(\gamma+1)/2}\e^{-s}ds
	= 2^{(\gamma+5)/2}\pi\Gamma\left(\frac{\gamma+3}{2}\right).
$$ 
Inserting the above estimate on $G(v)$ into \eqref{nu.pos} shows \eqref{Cnu}.

It remains to prove that $\na_v\nu_i\in L_v^\infty(\R^3)$. 
To this end, we compute
\begin{align}
  |\na\nu_i(v)| &= (2\pi)^{-3/2}\left|\sum_{j=1}^n c_{ij}\rho_{\infty,j}
	\int_{\R^3}\Phi'_{ij}(|v-v^*|)\cdot\frac{v-v^*}{|v-v^*|}
	\e^{-|v^*|^2/2}dv^*\right| \nonumber \\
	&\le (2\pi)^{-3/2}\sum_{j=1}^n c_{ij}\rho_{\infty,j}
	\int_{\R^3}|\Phi'_{ij}(|v-v^*|)|\e^{-|v^*|^2/2}dv^*. \label{aux3}
\end{align}
For given $R>0$, we decompose 
$$
  |\Phi'_{ij}(|v|)| = |\Phi'_{ij}(|v|)|\chi_{\{|v|<R\}}(v)
	+ |\Phi'_{ij}(|v|)|\chi_{\{|v|\ge R\}}(v).
$$
Assumption (A5) means that there exists $R>0$
such that 
$$
  |\Phi'_{ij}(|\cdot|)|\chi_{\{|\cdot|<R\}}\in L_v^1(\R^3)\quad\mbox{and}\quad
  |\Phi'_{ij}(|\cdot|)|\chi_{\{|\cdot|\ge R\}}\in L_v^\infty(\R^3). 
$$
Thus, the right-hand side of \eqref{aux3} is bounded since it can be written as the sum
of two terms, each of which is the convolution of an $L^1$ and an $L^\infty$ function.
This shows that $\na_v\nu_i\in L_v^\infty(\R^3)$. 
\end{proof}

\begin{remark}\rm
We observe that $\nu_i$ is generally not bounded since the kinetic part
$\Phi_{ij}(r)$ may grow like $r$ as $r\to\infty$. It is possible to show
that $\nu_i$ is bounded if $\Phi_{ij}$ is bounded. The unboundedness of $\nu_i$
implies that the spaces $L_v^2$ and $\H$ are not isomorphic.
\qed
\end{remark}



\section{Geometric properties of the spectrum}\label{sec.spec1}

In this section, we prove Theorem \ref{thm.ess} and the spectral-gap estimate 
\eqref{spec.gap} in Theorem \ref{thm.spec} by using arguments from 
functional analysis.

First, we study the essential spectrum of $L$ and $L-T$.
There exist several definitions of the essential spectrum of a linear operator.
Given a linear, closed and densely defined
operator $A:\mbox{Dom}(A)\subset X\to X$ on a Banach space $X$, we define
$$
  \sigma_{\rm ess}(A) = \{\lambda\in\C:A-\lambda I\mbox{ is not Fredholm}\}.
$$
We recall that a linear, closed, and densely defined operator $A$ is Fredholm 
if its range $\ran(A)$ is closed and both
its kernel and cokernel are finite-dimensional. For other definitions of
the essential spectrum, we refer to \cite{GuWe69}. The essential spectrum is
closed and conserved under compact perturbations, i.e., the bounded operators
$A$ and $B$ have the same essential spectrum if $A-B$ is compact
(Weyl's theorem; see \cite[Theorem S]{GuWe69}.

If $X$ is a Hilbert space and $A$ is selfadjoint, it holds
$\sigma_{\rm ess}(A)\subset\R$ and for given $\lambda\in\R$, we have $\lambda
\in\sigma_{\rm ess}(A)$ if and only if $A-\lambda I$ is not closed or the kernel
of $A-\lambda I$ is infinite dimensional. (This follows from the fact that
$\ran(A-\lambda I)^\perp=\ker(A-\lambda I)$ for closed, selfadjoint operators $A$
\cite[Chap.~V.3.1]{Kat80}.) Moreover, Weyl's criterion holds \cite[Lemma 5.17]{Tes09}:
$\lambda\in\sigma_{\rm ess}(A)$ if and only if $A-\lambda I$ admits a {\em singular
sequence}, i.e.\ a sequence $(f_k)\subset \mbox{Dom}(A)$ such that
(i) $\|f_k\|_X=1$ for all $k\in\N$; (ii) $\|(A-\lambda I)f_k\|_X\to 0$
as $k\to\infty$; and (iii) $(f_k)$ has no convergent subsequences in $X$.

We decompose $L$ as $L=K-\Lambda$, where 
$K=(K_1,\ldots,K_n)$, $\Lambda=(\Lambda_1,\ldots,\Lambda_n)$, and
\begin{align}
  K_i(f) &= \sum_{j=1}^n\int_{\R^3\times\S^2}B_{ij}M_i^{1/2}M_j^*(h_i'+h_j'^*-h_j^*)
	dv^*d\sigma, \nonumber \\
	\Lambda_i(f) &= \nu_i f_i, \quad 1\le i\le n, \label{Gamma}
\end{align}
and the collision frequencies $\nu_i$ are defined in \eqref{nu}. We recall 
from Lemma \ref{lem.nu} that they satisfy $\nu_i(v)\ge \nu_0>0$ and 
$|\nu_i(v)|\le C_\nu(1+|v|)$ for all $i=1,\ldots,n$ and $v\in\R^3$.

\begin{proof}[Proof of Theorem \ref{thm.ess}.]
Since $K$ is compact on $X=L_v^2$ \cite[Prop.~2]{BGPS13}, it follows 
that $\sigma_{\rm ess}(L)=\sigma_{\rm ess}(-\Lambda)=-\sigma_{\rm ess}(\Lambda)$.
Thus, we will first study the essential spectrum of $\Lambda$. 
The proof is divided into several steps.
\blue{Recall that $J=\cup_{i=1}^n\mbox{Im}(\nu_i)\subset[\nu_0,\infty)$.}

{\em Step 1: $\overline{J}\subset\sigma_{\rm ess}(\Lambda)$.} 
Let $\lambda\in J$. 
There exists $j\in\{1,\ldots,n\}$ and $\widehat v\in\R^3$ such that
$\lambda=\nu_j(\widehat v)$. We define the sequence $(f_k)\subset \D$ by
$$
  f_{k,i}(v) = (2\pi\sigma_k)^{-3/4}\exp\left(-\frac{|v-\widehat v|^2}{4\sigma_k}
	\right) \quad\mbox{if }i=j, \quad f_{k,i}(v)=0\quad\mbox{if }i\neq j,
$$
where $\sigma_k=1/k$, $k\in\N$.
Clearly, condition (i) for the singular sequence is satisfied. Furthermore,
\begin{align*}
  \|(\Lambda-\lambda I)f_k\|_{L_v^2}^2
	&= \sum_{i=1}^n\int_{\R^3}(\nu_i(v)-\lambda)^2f_{k,i}(v)^2 dv \\
	&= (2\pi\sigma_k)^{-3/2}\int_{\R^3}(\nu_i(v)-\nu_j(\widehat v))^2
	\exp\left(-\frac{|v-\widehat v|^2}{2\sigma_k}	\right)dv.
\end{align*}
The limit of a sequence of Gaussians with variance tending to zero converges
to the delta distribution $\delta_{\widehat v}$
(in the sense of distributions), which means that
$$
  (2\pi\sigma_k)^{-3/2}\int_{\R^3}u(v)\exp\left(
	-\frac{|v-\widehat v|^2}{2\sigma_k}	\right)dv \to u(\widehat v)
	\quad\mbox{as }k\to\infty
$$
for all functions $u\in C^0(\R^3)$ with polynomial growth at infinity.
Since $|\nu_i(v)|\le C_\nu(1+|v|)$, this condition is satisfied and we conclude
that $\|(\Lambda-\lambda I)f_k\|_{L_v^2}\to 0$ as $k\to\infty$, showing 
that condition (ii) holds.

Let us assume by contradiction that condition (iii) does not hold. 
Then there exists a subsequence $(f_{k_\ell})$ of $(f_k)$ that converges in
$L_v^2$ to some function $f\in L_v^2$. As a consequence, $|f_{k_\ell}|^2\to
|f|^2$ in $L^1_v$ as $\ell\to\infty$. In particular, $f\in L^2_v$. However,
the distributional limit $|f_{k_\ell}|^2\to \delta_{\widehat v}$
and the uniqueness of the limit imply that $\delta_{\widehat v}=|f|^2\in
L_v^1$, which is absurd. Thus, condition (iii) holds, and we infer that
$\lambda\in \sigma_{\rm ess}(\Lambda)$. Then, since $\sigma_{\rm ess}(\Lambda)$
is closed, $\overline{J}\subset\sigma_{\rm ess}(\Lambda)$.

{\em Step 2: $\sigma_{\rm ess}(\Lambda)\subset\overline{J}$.}
Let $´\lambda\in\R\backslash\overline{J}$. Then there exists a constant
$c>0$ such that for all $v\in\R^3$ and $i=1,\ldots,n$, $|\nu_i(v)-\lambda|\ge c$.
If $(f_k)\subset \D$ with $\|f_k\|_{L_v^2}=1$ for all $k\in\N$, we have
$$
  \|(\Lambda-\lambda I)f_k\|_{L_v^2}^2
	= \sum_{i=1}^n\int_{\R^3}(\nu_i(v)-\lambda)^2 f_{k,i}(v)^2 dv
	\ge c^2\sum_{i=1}^n\int_{\R^3}f_{k,i}(v)^2 dv = c^2 > 0
$$
for all $k\in\N$. Thus, condition (ii) cannot hold which implies that
$\lambda\not\in\sigma_{\rm ess}(\Lambda)$.

Steps 1 and 2 imply that $\sigma_{\rm ess}(\Lambda)=\overline{J}$. 

{\em Step 3: $\{\lambda\in\C:\Re(\lambda)\in J\}\subset
\sigma_{\rm ess}(\Lambda+T)$.} Let $\lambda\in \C$ be
such that $\Re(\lambda)\in J$. It follows from Step 1 that $\Re(\lambda)\in
\sigma_{\rm ess}(\Lambda)$. Since $\Lambda$ is selfadjoint on the Hilbert space
$L_v^2$, $\Lambda-\Re(\lambda)I$ is not closed or the kernel of 
$\Lambda-\Re(\lambda)I$ is infinite dimensional. As the operator
$\Lambda-\Re(\lambda)I$ is closed, its kernel must be infinite dimensional.
Therefore, there exists a sequence $(f_k)\subset L_v^2$ such that
$\Lambda(f_k)-\Re(\lambda)f_k=0$ and $(f_k,f_\ell)_{L_v^2}=\delta_{k\ell}$
for $k$, $\ell\in\N$. Let us define $\phi(x,v)=\exp(\I\Im(\lambda)x\cdot v/|v|^2)$
and $g_k=\phi f_k\in L_{x,v}^2$. Since $|\phi|=1$, we have
\begin{equation}\label{gg}
  (g_k,g_\ell)_{L_{x,v}^2} = (f_k,f_\ell)_{L_v^2} = \delta_{k\ell} \quad
	\mbox{for }k,\ell\in\N.
\end{equation}
Furthermore, $\phi\in \mbox{Dom}(T)$ and $T(\phi)=\I\Im(\lambda)\phi$ for $v\neq 0$,
and thus,
$$
  (\Lambda + T - \lambda I)g_k = \phi(\Lambda-\Re(\lambda)I)f_k
	+ f_k(T-\I\Im(\lambda)I)\phi = 0,
$$
which shows that $g_k\in \ker(\Lambda+T-\lambda I)$ for $k\in\N$.
This fact, together with relation \eqref{gg}, implies that $\ker(\Lambda+T-\lambda I)$
is infinite dimensional. As a consequence, $\Lambda+T-\lambda I$ is not
Fredholm and $\lambda\in\sigma_{\rm ess}(\Lambda+T)$, which proves the claim.

{\em Step 4: $\{\lambda\in\C:\Re(\lambda)\not\in\overline{J}\}\subset
\rho(\Lambda+T)$.} Clearly, this gives
$$
  \sigma_{\rm ess}(\Lambda+T)\subset\sigma(\Lambda+T)\subset
	\{\lambda\in\C:\Re(\lambda)\in\overline{J}\}.
$$
Let $\lambda\in\C$ be such that $\Re(\lambda)\in\R\backslash\overline{J}$.
We show first that $\ker(\Lambda+T-\lambda I)=\{0\}$. We assume by contradiction
that there exists $f\in \mbox{Dom}(\Lambda+T)$ satisfying $\|f\|_{L_{x,v}^2}>0$ and
$(\Lambda+T-\lambda I)f=0$. In particular, there is an index
$\ell\in\{1,\ldots,n\}$ such that $\int_{\R^3}\int_{\T^3}f_\ell^2 dxdv>0$.
Then, multiplying $\nu_\ell f_\ell+T(f_\ell)=\lambda f_\ell$ by $\overline{f}_\ell$
(the complex conjugate of $f_\ell$) and integrating in $\T^3\times\R^3$, we obtain
\begin{equation}\label{int1}
  \int_{\R^3}\int_{\T^3}\nu_\ell|f_\ell|^2 dxdv 
	+ \int_{\R^3}\int_{\T^3}\overline{f}_\ell v\cdot\na_x f_\ell dxdv
	= \lambda\int_{\R^3}\int_{\T^3}|f_\ell|^2dxdv.
\end{equation}
By the divergence theorem, the real part of the second integral vanishes,
\begin{align}
  2\Re\int_{\R^3}\int_{\T^3}\overline{f}_\ell v\cdot\na_x f_\ell dxdv
	&= \int_{\R^3}\int_{\T^3}\big(\overline{f}_\ell v\cdot\na_x f_\ell
	+ f_\ell v\cdot\na_x\overline{f}_\ell\big)dxdv \label{antisymm} \\
	&= \int_{\R^3}\int_{\T^3}v\cdot\na_x|f_\ell|^2 dxdv = 0. \nonumber
\end{align}
Then, taking the real part of \eqref{int1}, we infer that
$$
  \Re(\lambda) = \frac{\int_{\R^3}\int_{\T^3}\nu_\ell|f_\ell|^2 dxdv}{
	\int_{\R^3}\int_{\T^3}|f_\ell|^2 dxdv}.
$$
Consequently, $\inf_{\R^3}\nu_\ell \le\Re(\lambda)\le\sup_{\R^3}\nu_\ell$
and, thanks to the continuity of $\nu_\ell$, $\Re(\lambda)\in
\overline{\mbox{Im}(\nu_\ell)}\subset\overline{J}$, which is a contradiction.
Thus, $\ker(\Lambda+T-\lambda I)=\{0\}$. Similarly, we can show that
$\ker((\Lambda+T-\lambda I)^*) = 
\ker(\Lambda+T^*-\overline{\lambda}I)=\{0\}$ as well.

The operator $L-T$ is closed \cite[Theorems 2.2.1 and 2.2.2]{UkYa06}. Thus,
the boundedness of the compact operator $K$ and the stability of
closedness under bounded perturbations \cite[Chap.~III, Problem 5.6]{Kat80}
imply that $\Lambda+T=K-(L-T)$ is closed (and also densely defined).
Hence, $\ran(\Lambda+T-\lambda I)^\perp = \ker((\Lambda+T-\lambda I)^*) = \{0\}$,
meaning that $\Lambda+T-\lambda I$ is invertible. If $f\in L_{x,v}^2$ is given,
there exists $u\in \mbox{Dom}(\Lambda+T)$ such that $(\Lambda+T-\lambda I)u=f$,
which translates into
\begin{equation}\label{nuj}
  (\nu_j-\Re(\lambda))u_j + (T-\I\Im(\lambda))u_j = f, \quad
	j=1,\ldots,n.
\end{equation}
We point out that, since $\nu_j$ is continuous, $\mbox{Im}(\nu_j)$ is an interval
(or a point, in case that $\nu_j$ is constant). This fact and the assumption
$\Re(\lambda)\not\in\overline{J}$ imply that either $\nu_j-\Re(\lambda)>0$
in $\R^3$ or $\nu_j-\Re(\lambda)<0$ in $\R^3$. This means that the sign $s_j$ of
$\nu_j-\Re(\lambda)$ is constant in $\R^3$, for $j=1,\ldots,n$. By multiplying
\eqref{nuj} by $s_j\overline{u}_j$, integrating over $\T^3\times\R^3$,
taking the real part, and summing over $j=1,\ldots,n$, we find that
$$
  \sum_{j=1}^n\int_{\R^3}\int_{\T^3}|\nu_j-\Re(\lambda)||u_j|^2 dxdv
	= \frac12\sum_{j=1}^n s_j\int_{\R^3}\int_{\T^3}\big(\overline{u}_j f_j
	+ u_j\overline{f}_j\big)dxdv.
$$
The (real part of the) second term in \eqref{nuj} vanishes after integration;
see \eqref{antisymm}. \blue{Since $\lambda\in\R\backslash\overline{J}$,
by definition of $J$,} there exists $c_\lambda>0$ such that
$|\nu_j-\Re(\lambda)|\ge c_\lambda$ in $\R^3$ for all $j=1,\ldots,n$. Then
the Cauchy-Schwarz inequality shows that $\|u\|_{L_{x,v}^2}\le c_\lambda^{-1}
\|f\|_{L_{x,v}^2}$. This means that $(\Lambda+T-\lambda I)^{-1}$ is bounded,
so $\lambda\in\rho(\Lambda+T)$. 

Steps 3 and 4 show that $\sigma_{\rm ess}(\Lambda+T)=\{\lambda\in\C:
\Re(\lambda)\in\overline{J}\}$.

{\em Step 5:
  $\sigma_{\rm ess}(-\Lambda-T)=\sigma_{\rm ess}(K-\Lambda-T)$.}  The
operator $K$ is compact on $L_v^2$ but not on $L_{x,v}^2$, so the
claim does not follow from the original form of Weyl's
theorem. Instead we will employ the fact that the essential spectrum
is conserved under a relatively compact perturbation \cite[Section
IV.5.6, Theorem 5.35]{Kat80}. More precisely, we prove that $K$ is
relatively compact with respect to $\Lambda+T$, i.e.,
$B_z:=K(\Lambda+T-zI)^{-1}$ is compact on $L_{x,v}^2$ for some
$z\in\C$ with $\Re(z)\in\R\backslash\overline{J}$. (Notice that by
Step 4, $z\in\rho(\Lambda+T)$.) Then
\begin{equation}\label{comp}
  \sigma_{\rm ess}(K-\Lambda-T) = \sigma_{\rm ess}(-\Lambda-T)
	= \{\lambda\in\C:\Re(\lambda)\in-J\}.
\end{equation}
The second identity is a consequence of Steps 3 and 4. 

To prove the compactness
of $B_z$, we introduce the space $W:=\ell^2(\Z^3;L_v^2)$ of sequences
$f=(f_m)\subset L_v^2$ with the canonical norm $\|f\|_W = (\sum_{m\in\Z^3}
\|f_m\|_{L_v^2}^2)^{1/2}$. Clearly, $W$ is a Hilbert space with the scalar product
$(f,g)_W = \sum_{m\in\Z^3}(f_m,g_m)_{L_v^2}$. Furthermore, we introduce the
Fourier mapping $F:L_{x,v}^2(\T^3\times\R^3)\to W$ by
$$
  F(f) = (\widehat f_m), \quad \widehat f_m(v) = \int_{\T^3} e^{-2\pi\I m\cdot x}
	f(x,v)dx\quad\mbox{for }m\in\Z^3,\, v\in\R^3.
$$
This mapping is bounded, invertible, and has a bounded inverse. We wish to 
show that $\widehat B_z=FB_zF^{-1}:W\to W$ is compact. Then also
$B_z=F^{-1}\widehat B_z F$ is compact as a composition of a compact
and two bounded operators. This idea is due to Ukai; see e.g.\ 
\cite[Section 2.2.1]{UkYa06}. 

Since $K$ and $\Lambda$ do not depend on $x$, it holds that
$\widehat B_z = K(\Lambda+\widehat T-z)^{-1}$, where $\widehat T=2\pi\I v\cdot m$.
Let $(f^{(k)})=(f^{(k)}_m)\subset W$ be a bounded sequence in $W$, i.e., there
exists $c_0>0$ such that for all $k\in\N$,
\begin{equation}\label{c0}
  \|f^{(k)}\|_W^2 = \sum_{m\in\Z^3}\|f^{(k)}_m\|_{L_v^2}^2 \le c_0.
\end{equation}
As $\Re(z)\in \R\backslash\overline{J}$, there is a constant $c_z>0$ such that
for all $i=1,\ldots,n$ and $v\in\R^3$, $|\nu_i(v)+2\pi\I v\cdot m-z|\ge c_z$.
Thus,
\begin{align*}
  \|(\Lambda+\widehat T-z)^{-1}f_m^{(k)}\|_{L_v^2}^2
	&= \sum_{i=1}^n\int_{\R^3}\big|(\nu_i(v)+2\pi\I v\cdot m-z)^{-1}f_{m,i}^{(k)}
	\big|^2 dv \\
	&\le c_z^{-2}\sum_{i=1}^n\int_{\R^3}|f_{m,i}^{(k)}|^2 dv
	= c_z^{-2}\|f_m^{(k)}\|_{L_v^2}^2.
\end{align*}
Summing these inequalities over $m\in\Z^3$, we infer that
$$
  \|(\Lambda+\widehat T-z)^{-1}f^{(k)}\|_W^2
	\le c_z^{-2}\|f^{(k)}\|_W^2 \le c_0c_z^{-2}.
$$
Consequently, the sequence $g^{(k)}:=(\Lambda+\widehat T-z)^{-1}f^{(k)}$ is bounded
in $W$. In particular, for any $s\in\Z^3$, $\|g_s^{(k)}\|_{L_v^2}^2
\le \sum_{m\in\Z^3}\|g_m^{(k)}\|_{L_v^2}^2=\|g^{(k)}\|_W^2 \le c_0c_z^{-2}$.
Hence, for any $s\in\Z^3$, the sequence $(g_s^{(k)})\subset L_v^2$ is bounded
in $L_v^2$. Since $K:L_v^2\to L_v^2$ is compact and $\Z^3$ is countable,
we may apply Cantor's diagonal argument to find a subsequence $(g^{(k_\ell)})$
of $(g^{(k)})$ such that $(K(g_m^{(k_\ell)}))$ is convergent in $L_v^2$
as $\ell\to\infty$, for all $m\in\Z^3$.

We will show that $(\widehat B_z(f^{(k_\ell)}))$ is a Cauchy sequence in $W$.
To this end, let $\ell$, $s$, $N\in\N$. We write
\begin{align}
  \|\widehat B_z(f^{(k_\ell)}) & - \widehat B_z(f^{(k_s)})\|_W^2
	= \sum_{m\in\Z^3}\|K(g_m^{(k_\ell)}) - K(g_m^{(k_s)})\|_{L_v^2}^2 \label{Bz} \\
	&= \sum_{|m|\le N}\|K(g_m^{(k_\ell)}) - K(g_m^{(k_s)})\|_{L_v^2}^2
	+ \sum_{|m|>N}\|K(g_m^{(k_\ell)}) - K(g_m^{(k_s)})\|_{L_v^2}^2, \nonumber
\end{align}
where $|m|=\sum_{i=1}^3|m_i|$ for all $m\in\Z^3$. First, we consider the
second sum on the right-hand side. Denote by $\|\cdot\|_{\mathscr{L}(L_v^2)}$
the norm in the space of linear bounded operators on $L_v^2$. By the definition
of $g_m^{(k)}$, we obtain
\begin{align}
  \sum_{|m|>N} & \|K(g_m^{(k_\ell)}) - K(g_m^{(k_s)})\|_{L_v^2}^2
	= \sum_{|m|>N}\|K(\Lambda+2\pi\I v\cdot m-z)^{-1}(f_m^{(k_\ell)} - f_m^{(k_s)})
	\|_{L_v^2}^2 \label{int2} \\
	&\le 2\sum_{|m|>N}\|K(\Lambda+2\pi\I v\cdot m-z)^{-1}\|_{\mathscr{L}(L_v^2)}^2
	\big(\|f_m^{(k_\ell)}\|_{L_v^2}^2 + \|f_m^{(k_s)}\|_{L_v^2}^2\big). \nonumber
\end{align}
For the operator norm, we employ Prop.~2.2.6 in \cite{UkYa06}, which can be applied
since $\Re(z)\in\R\backslash\overline{J}$:
$$
  \|K(\Lambda+2\pi\I v\cdot m-z)^{-1}\|_{\mathscr{L}(L_v^2)}^2
	\le c_1(1+|m|)^{-\alpha} \quad\mbox{for all }m\in\Z^3
$$
for some suitable constant $c_1>0$ (depending on $z$) and a suitable exponent
$\alpha\in(0,1)$ (actually, $\alpha=4/13$). Let $0<\beta<2\alpha/3$. 
By H\"older's inequality and \eqref{c0}, we estimate
\begin{align*}
  \sum_{|m|>N} & \|K(\Lambda+2\pi\I v\cdot m-z)^{-1}\|_{\mathscr{L}(L_v^2)}^2
	\|f_m^{(k)}\|_{L_v^2}^2
	\le c_0^{\beta/2}c_1\sum_{|m|>N}(1+|m|)^{-\alpha}\|f_m^{(k)}\|_{L_v^2}^{2-\beta} \\
	&\le c_0^{\beta/2}c_1\bigg(\sum_{|m|>N}(1+|m|)^{-2\alpha/\beta}\bigg)^{\beta/2}
	\bigg(\sum_{|m|>N}\|f_m^{(k)}\|_{L_v^2}^2\bigg)^{1-\beta/2} \\
	&\le c_0 c_1\bigg(\sum_{|m|>N}(1+|m|)^{-2\alpha/\beta}\bigg)^{\beta/2}.
\end{align*}
Using this estimate in \eqref{int2}, it follows that
$$
  \sup_{\ell,s\in\N}\sum_{|m|>N}\|K(g_m^{(k_\ell)}) - K(g_m^{(k_s)})\|_{L_v^2}^2
	\le 2c_0c_1\bigg(\sum_{|m|>N}(1+|m|)^{-2\alpha/\beta}\bigg)^{\beta/2}.
$$
The choice of $\beta$ implies that $2\alpha/\beta>3$ and hence, the sum over
$|m|>N$ is finite. In particular, $\sum_{|m|>N}(1+|m|)^{-2\alpha/\beta}\to 0$
as $N\to\infty$. As a consequence, for given $\eps>0$, there exists $N_\eps\in\N$
such that 
$$
  \sup_{\ell,s\in\N}\sum_{|m|>N_\eps}\|K(g_m^{(k_\ell)}) - K(g_m^{(k_s)})\|_{L_v^2}^2
	< \frac{\eps}{2}.
$$
Finally, since $(K(g_m^{(k_\ell)}))$ is convergent in $L_v^2$ for all
$m\in\Z^3$, there is a number $\eta=\eta(\eps)>0$ such that for all $\ell$, $s>\eta$,
$$
  \sum_{|m|\le N_\eps}\|K(g_m^{(k_\ell)}) - K(g_m^{(k_s)})\|_{L_v^2}^2 
	< \frac{\eps}{2}.
$$
Thus, choosing $N=N_\eps$ in \eqref{Bz}, we deduce that $(\widehat B_z(f^{(k_s)}))$
is a Cauchy sequence in the Hilbert space $W$ and consequently, it is convergent.
This shows that $\widehat B_z:W\to W$ is a compact operator and \eqref{comp} holds.
This finishes the proof of Theorem \ref{thm.ess}.
\end{proof}

Next, we show the spectral-gap estimate for the linearized collision operator 
$L=K-\Lambda$, i.e.\ the first statement of Theorem \ref{thm.spec}. 
Since $K$ is compact on $L_v^2$, it remains to prove that 
$\Lambda:\D \subset L_v^2\to L_v^2$ is coercive. 

\begin{lemma}\label{lem.Lambda}
Let (A1)-(A4) hold. Then the embedding $\H\hookrightarrow L_v^2$
is continuous and $\Lambda:\D\to L_v^2$,
defined in \eqref{Gamma}, is a linear unbounded operator with the property
\begin{equation}\label{coerc}
  (f,\Lambda(f))_{L_v^2} = \|f\|_\H^2 \ge C\|f\|_{L_v^2}^2\quad\mbox{for }f\in\H
\end{equation}
for some $C>0$. Moreover, $\Lambda$ can be extended by density to a
linear bounded operator $\Lambda:\H\to\H'$, where $\H'$ is the
  dual of $\H$ with respect to the $L^2_v$ scalar product.  In
particular, the mapping $\H\to\R$,
$f\mapsto\langle \Lambda(f),f\rangle$ is continuous, where
$\langle\cdot,\cdot\rangle$ denotes the duality pairing between $\H'$
and $\H$.
\end{lemma}

\begin{proof}
The strict positivity of $\nu_i$ in $\R^3$ (see Lemma \ref{lem.nu})
implies that the embedding $\H\hookrightarrow L_v^2$ is continuous.
Then the definitions of $\Lambda_i$ and $\H$ show that for all $f\in\H$,
\eqref{coerc} holds.
For given $f\in\H$, the element $\Lambda(f)=(f_1\nu_1,\ldots,f_n\nu_n)$ 
can be identified with the linear bounded operator $\H\to\R$, 
$g\mapsto\sum_{i=1}^n\int_{\R^3}g_if_i\nu_idv$ and consequently, $\Lambda(f)\in\H'$. 
It is immediate to see that $\|\Lambda(f)\|_{\H'}=\|f\|_\H$, so that 
$\Lambda:\H\to\H'$ is isometric and thus bounded. 
Moreover, it follows that 
$\H\to\R$, $f\mapsto\langle \Lambda(f),f\rangle$, is continuous.
\end{proof}

The following result provides a spectral gap for general operators
which decompose into a compact and a coercive part.

\begin{lemma}\label{lem.spec.gap}
Let $\H_0$ and $\H$ be Hilbert spaces such that $\H\hookrightarrow \H_0$ continuously
and let $L:\H\to\H'$ be a linear bounded operator such that $L=K-\Lambda$ 
with linear bounded operators $\Lambda:\H\to\H'$ and $K:\H_0\to\H_0$.
Furthermore, assume that
\begin{enumerate}[\rm (i)]
\item for all $f\in\H$, $\langle L(f),f\rangle\le 0$ with equality holding
if and only if $f\in\ker(L)$;
\item the operator $K:\H_0\to\H_0$ is compact;
\item there exists $C_0>0$ such that for all $f\in\H$, $\langle\Lambda(f),f\rangle
\ge C_0\|f\|_\H^2$.
\end{enumerate}
Then there exists a constant $C_1>0$ such that 
$$
  -\langle L(f),f\rangle \ge C_1\|f\|_\H^2 \quad\mbox{for all }
	f\in\H\cap\ker(L)^\perp.
$$
\end{lemma}

\begin{proof}
  We argue by contradiction. Let $(f_n)\subset\H\cap\ker(L)^\perp$ be
  a sequence such that $\|f_n\|_\H=1$ for $n\ge 1$ but
  $\langle L(f_n),f_n\rangle\to 0$ as $n\to\infty$. Since $(f_n)$ is
  bounded in the Hilbert space $\H$, there exists a subsequence, which
  is not relabeled, such that $f_n\rightharpoonup f$ weakly in
  $\H$. Because of the continuous embedding $\H\hookrightarrow\H_0$,
  also $f_n\rightharpoonup f$ weakly in $\H_0$. Since
  $f_n\in\ker(L)^\perp$ and $\ker(L)^\perp$ is weakly closed by
  Mazur's lemma, $f\in\ker(L)^\perp$. As the operator $K:\H_0\to\H_0$
  is compact, by hypothesis (ii), the weak convergence of $(f_n)$ in
  $\H_0$ implies that $K(f_n)\to K(f)$ strongly in $\H_0$. Hence,
  $(f_n,K(f_n))_{\H_0}\to (f,K(f))_{\H_0}$.  Since $\Lambda:\H\to\H'$
  is bounded, the mapping $G:\H\to\R$,
  $f\mapsto\langle\Lambda(f),f\rangle$, is continuous. The linearity
  of $\Lambda$ and property (iii) imply that $G$ is also convex. Thus,
  $G$ is weakly lower semicontinuous \cite[Corollary
  3.9]{Bre10}. Therefore,
$$
  -\langle L(f),f\rangle = \langle\Lambda(f),f\rangle - (K(f),f)_{\H_0}
	\le \liminf_{n\to\infty}\big(\langle \Lambda(f_n),f_n\rangle 
	- (K(f_n),f_n)_{\H_0}\big)
	= 0,
$$
because $\langle L(f_n),f_n\rangle\to 0$ as $n\to\infty$ by assumption. We infer from
hypothesis (i) that $f\in\ker(L)$. But also $f\in\ker(L)^\perp$, so $f=0$. Then,
by hypothesis (iii),
$$
  0 < C_0 = C_0\|f_n\|_\H^2 \le \langle\Lambda(f_n),f_n\rangle
	= (K(f_n),f_n)_{\H_0} - \langle L(f_n),f_n\rangle\to 0,
$$
which is a contradiction.
\end{proof}

Let $\H_0=L_v^2$. By \cite[Prop.~2]{BGPS13}, assumption (ii) of Lemma
\ref{lem.spec.gap} holds. Furthermore, Lemma \ref{lem.Lambda} shows that 
(iii) holds true. Assumption (i) is a consequence of Lemma \ref{lem.H}. 
\blue{Let $f\in\D\subset\H$ und set 
$\widetilde f=f-\varPi^L(f)\in \ker(L)^\perp$. Then
$$
  -\langle L(f),f\rangle = \langle L(\widetilde f),\widetilde f\rangle
	\ge C\|\widetilde f\|_\H^2 = C\|f-\varPi^L(f)\|_\H^2,
$$}%
since $L(f)\in L_v^2$ and $\langle L(f),f\rangle = (f,L(f))_{L_v^2}$ for 
$f\in \D$. This proves the first statement in Theorem \ref{thm.spec}.


\section{Explicit spectral gap estimate}\label{sec.spec2}

We present a second proof of the spectral-gap estimate \eqref{spec.gap}
with explicit constants. The idea is to decompose the collision operator
$L$ into a mono-species and a multi-species part and to exploit the fact
that the conservation properties of $L$ are different from those of the mono-species
part $L^m$. Let assumptions (A1)-(A4) hold.

\subsection{Decomposition}

We decompose $L=L^m+L^b$, where $L^m=(L_1^m,\ldots,L_n^m)$, $L^b=(L_1^b,\ldots,L_n^b)$,
and
\begin{equation}\label{Lmb}
  L_i^m(f_i) = L_{ii}(f_i,f_i), \quad L_i^b(f) = \sum_{j\neq i}L_{ij}(f_i,f_j).
\end{equation}
Denoting by $\varPi^m$ the orthogonal projection onto $\ker(L^m)$ (with respect to the
scalar product in $L_v^2$), we can decompose $f$ according to
\begin{equation}\label{fdecomp}
  f = f^\parallel + f^\perp, \quad\mbox{where }f^\parallel:=\varPi^m(f), \quad
	f^\perp:=f-f^\parallel.
\end{equation}
Lemma \ref{lem.H} shows that
\begin{align}
  f\in \ker(L)\mbox{ if and only if }
	f_i=M_i^{1/2}(\alpha_i+u\cdot v+e|v|^2)\ \mbox{for }\alpha_i,e\in\R,\ 
	u\in\R^3, \label{kerL} \\
	f\in \ker(L^m)\mbox{ if and only if }
		f_i=M_i^{1/2}(\alpha_i+u_i\cdot v+e_i|v|^2)\ \mbox{for }\alpha_i,e_i\in\R,\ 
	u_i\in\R^3, \label{kerLm}
\end{align}
and $f^\parallel$ has clearly the form \eqref{kerLm}.

For later use, we define the following bilinear forms
\begin{align}
  -(f,L^m(f))_{L_v^2} &= \frac14\sum_{i=1}^n\int_{\R^6\times\S^2} B_{ii}
	\Delta_i[h_i]^2 M_iM_i^* dvdv^*d\sigma, \label{Delta} \\
	-(f,L^b(f))_{L_v^2} &= \frac14\sum_{i=1}^n\sum_{j\neq i}\int_{\R^6\times\S^2}
	B_{ij}A_{ij}[h_i,h_j]^2 M_iM_j^*dvdv^*d\sigma, \label{Aij}
\end{align}
where $h_i = M_i^{-1/2}f_i$ and
$$
  \Delta_i[h_i] := h_i'+h_i'^*-h_i-h_i^*, \quad
	A_{ij}[h_i,h_j] := h_i'+h_j'^*-h_i-h_j^*.
$$


\subsection{Spectral-gap estimate for $L^m$}\label{sec.specLm}

Our starting point is the fact that the mono-species collision
operator $L^m$ has an explicitly computable spectral gap. A
spectral-gap estimate for the linearized collision operator with $n=1$
was proved in \cite[Theorem 6.1, Remark 1]{Mou04}:
\begin{equation}\label{spec1}
  \frac14\int_{\R^6\times\S^2}B_{ii}\Delta_i[h_i]^2 M_i M_i^*
	dvdv^*d\sigma \ge 
	\frac{\lambda_m}{\rho_{\infty,i}}\int_{\R^3}(f_i-\varPi^m(f_i))^2\nu_{ii}dv,
\end{equation}
where $\lambda_m=\lambda_m(\gamma,C_1,C^b)>0$, 
only depending on $\gamma$, $C_1$, and $C^b$
(see (A3)-(A4)), can be computed explicitly, 
$$
  \nu_{ii}(v) := \int_{\R^3\times\S^2}B_{ii}(|v-v^*|,\cos\vartheta)
	M_i^* dv^*d\sigma,
$$
and $i\in\{1,\ldots,n\}$ is fixed. This yields the following estimate for $L^m$,
where we recall that the space $\H$ is defined in \eqref{H}.
\begin{lemma}\label{lem.specLm}
With $L^m$ defined in \eqref{Lmb}, we have
$$
  -(f,L^m(f))_{L_v^2} \ge C^m \|f-\varPi^m(f)\|_\H^2 \quad\mbox{for all }
	f\in\mbox{\emph{Dom}}(L^m),
$$
where 
\begin{equation}\label{Cm}
  C^m = \frac{\lambda^m(\gamma,C_1,C^b)}{\beta\rho_{\infty}},
\end{equation}
and $\lambda^m=\lambda^m(\gamma,C_1,C^b)$ is given in \eqref{spec1}.
\end{lemma}
\begin{proof}
We sum \eqref{spec1} over $i=1,\ldots,n$ and employ \eqref{Delta} to obtain
\begin{equation}\label{fLmf}
  -(f,L^m(f))_{L_v^2} \ge \lambda^m\sum_{i=1}^n\int_{\R^3}
	(f_i-\varPi^m(f))^2\frac{\nu_{ii}}{\rho_{\infty,i}}dv.
\end{equation}
It remains to estimate $\nu_{ii}$ in terms of $\nu_i$, defined in \eqref{nu}.
The definition of $M_i$ implies that $M_j=(\rho_{\infty,j}/\rho_{\infty,i})M_i$. 
This fact, as well as definition \eqref{nu} of $\nu_i$, the lower bound \eqref{Cnu},
and assumption (A6) give
$$
  \nu_i = \sum_{j=1}^n\frac{\rho_{\infty,j}}{\rho_{\infty,i}}
	\int_{\R^3}B_{ij}M_i^* dv^*d\sigma
	\le \beta \sum_{j=1}^n\frac{\rho_{\infty,j}}{\rho_{\infty,i}}
	\int_{\R^3}B_{ii}M_i^* dv^*d\sigma 
	= \frac{\beta\rho_\infty}{\rho_{\infty,i}}\nu_{ii}.
$$
We conclude that $\nu_{ii}/\rho_{\infty,i}\ge\nu_i/(\beta\rho_\infty)$, and inserting
this bound into \eqref{fLmf} yields the result.
\end{proof}

Lemma \eqref{lem.specLm} and the inequality $-(f,L^b(f))_{L_2^v}\ge 0$
immediately show that
$$
  -(f,L(f))_{L_v^2} \ge C^m\|f-\varPi^m(f)\|_\H^2 \quad\mbox{for all }
	f\in\D.
$$
However, we need the projection onto $\ker(L)^\perp$ instead of $\ker(L^m)^\perp$,
which is contained in $\ker(L)^\perp$. Therefore, we will exploit the part
$-(f,L^b(f))_{L_v^2}$ to derive a sharper estimate.


\subsection{Absorption of the orthogonal parts}\label{sec.ortho}

We prove that the contribution $f^\perp$ (introduced in \eqref{fdecomp})
in the term $-(f,L^b(f))_{L_v^2}
= -(f^\parallel+f^\perp,L^b(f^\parallel+f^\perp))_{L_v^2}$ can be absorbed
by the $\H$ norm of $f^\perp$.

\begin{lemma}\label{lem.ortho}
Let $\eta=\min\{1,C^m/8\}$, where $C^m>0$ is given in Lemma \ref{lem.specLm}.
Then, for all $f\in\D$,
$$
  -(f,L(f))_{L_v^2} \ge (C^m-4\eta)\|f-f^\parallel\|_\H^2
	- \frac{\eta}{2}(f^\parallel,L^b(f^\parallel))_{L_v^2},
$$
where $f^\parallel=\varPi^m(f)$ is the projection onto $\ker(L^m)^\perp$.
\end{lemma}

\begin{proof}
By Lemma \ref{lem.specLm}, we find that
\begin{equation}\label{fLf}
  -(f,L(f))_{L_v^2} \ge C^m\|f-f^\parallel\|_\H^2 - (f,L^b(f))_{L_v^2}
	\ge C^m\|f-f^\parallel\|_\H^2 - \eta(f,L^b(f))_{L_v^2},
\end{equation}
since $-(1-\eta)(f,L^b(f))_{L_v^2}\ge 0$ for $\eta\in(0,1]$. We estimate first
the expression $A_{ij}[h_i,h_j]$ in definition \eqref{Aij}, writing
$h_i^\parallel=M_i^{-1/2}f_i^\parallel$ and $h_i^\perp=M_i^{-1/2}f_i^\perp$,
\begin{align*}
  A_{ij}[h_i,h_j]^2 &= \big(A_{ij}[h_i^\parallel,h_j^\parallel]
	+ A_{ij}[h_i^\perp,h_j^\perp]\big)^2 \\
	&= A_{ij}[h_i^\parallel,h_j^\parallel]^2 + A_{ij}[h_i^\perp,h_j^\perp]^2
	+ 2A_{ij}[h_i^\parallel,h_j^\parallel]A_{ij}[h_i^\perp,h_j^\perp] \\
	&\ge \frac12A_{ij}[h_i^\parallel,h_j^\parallel]^2
	- A_{ij}[h_i^\perp,h_j^\perp]^2.
\end{align*}
Inserting this estimate into \eqref{Aij} and \eqref{fLf} gives
\begin{align}
  -(f,L(f))_{L_v^2} &\ge C^m\|f^\perp\|_\H^2
	+ \frac{\eta}{8}\sum_{i=1}^n\sum_{j\neq i}\int_{\R^6\times\S^2}
	B_{ij}A_{ij}[h_i^\parallel,h_j^\parallel]^2M_iM_j^*dvdv^*d\sigma \nonumber \\
	&\phantom{xx}{}- \frac{\eta}{4}\sum_{i=1}^n\sum_{j\neq i}\int_{\R^6\times\S^2}
	B_{ij}A_{ij}[h_i^\perp,h_j^\perp]^2 M_iM_j^*dvdv^*d\sigma. \label{fLf2}
\end{align}

We claim that the last term on the right-hand side can be estimated from below
by $\|f^\perp\|_\H^2$, up to a small factor. For this, we employ the invariance
properties of $B_{ij}$ and the identity $M_iM_j^*=M_i'M_j'^*$:
\begin{align*}
  \int_{\R^6\times\S^2} & B_{ij}A_{ij}[h_i^\perp,h_j^\perp]^2 M_iM_j^*dvdv^*d\sigma \\
	&\le 4\int_{\R^6\times\S^2}B_{ij}\big(((h_i^\perp)')^2 + ((h_j^\perp)'^*)^2
	+ (h_i^\perp)^2 + ((h_j^\perp)^*)^2\big)M_iM_j^* dvdv^*d\sigma \\
	&\le 16\int_{\R^6\times\S^2}B_{ij}(h_i^\perp)^2 M_iM_j^* dvdv^*d\sigma
	= 16\int_{\R^6\times\S^2}B_{ij}(f_i^\perp)^2 M_j^* dvdv^*d\sigma.
\end{align*}
Thus, the last term on the right-hand side of \eqref{fLf2} can be estimated as
\begin{align*}
  - &\frac{\eta}{4} \sum_{i=1}^n\sum_{j\neq i} \int_{\R^6\times\S^2}
	B_{ij}A_{ij}[h_i^\perp,h_j^\perp]^2 M_iM_j^*dvdv^*d\sigma \\
	&\ge -4\eta\sum_{i=1}^n\sum_{j\neq i}\int_{\R^6\times\S^2} B_{ij}
	(f_i^\perp)^2 M_j^*dvdv^*d\sigma 
	\ge -4\eta\sum_{i=1}^n\int_{\R^3}(f_i^\perp)^2\nu_i dv
	= -4\eta\|f^\perp\|_\H^2,
\end{align*}
taking into account definition \eqref{nu} of $\nu_i$. We infer from \eqref{fLf2}
that
$$
  -(f,L(f))_{L_v^2} \ge (C^m-4\eta)\|f-f^\parallel\|_\H^2
	+ \frac{\eta}{8}\sum_{i=1}^n\sum_{j\neq i}\int_{\R^6\times\S^2}
	B_{ij}A_{ij}[h_i^\parallel,h_j^\parallel]^2M_iM_j^*dvdv^*d\sigma,
$$
and definition \eqref{Aij} yields the conclusion.
\end{proof}


\subsection{Estimate for the remaining part}\label{sec.remain}

It remains to estimate the term $-(f^\parallel,L^b(f^\parallel))_{L_v^2}$.

\begin{lemma}\label{lem.Lb}
For $f^\parallel\in\ker(L^m)$,
i.e.\ $f_i^\parallel=M_i^{1/2}(\alpha_i+u_i\cdot v+e_i|v|^2)$ for some
$\alpha_i$, $e_i\in\R$ and $u_i\in\R^3$, we have
$$
  -(f^\parallel,L^b(f^\parallel))_{L_v^2} \ge \frac{D^b}{4}\sum_{i,j=1}^n
	\big(|u_i-u_j|^2+(e_i-e_j)^2\big),
$$
where $D^b>0$ is defined in \eqref{D}.
\end{lemma}

\begin{proof}
Thanks to the momentum and energy conservation, we obtain differences of the
momenta and energies, which will be crucial in the following:
\begin{align*}
  u_i\cdot v' + u_j\cdot v'^* - u_i\cdot v - u_j\cdot v^* &= (u_i-u_j)\cdot(v'-v), \\
  e_i|v'|^2 + e_j|v'^*|^2 - e_i|v|^2 - e_j|v^*|^2 &= (e_i-e_j)(|v'|^2-|v|^2).
\end{align*}
Using these identities in $A_{ij}[h_i^\parallel,h_j^\parallel]$, where
$h_i^\parallel=\alpha_i+u_i\cdot v+e_i|v|^2$, we find that
\begin{align*}
  -&(f^\parallel,L^b(f^\parallel))_{L_v^2}
	= \frac14\sum_{i=1}^n\sum_{j\neq i}\int_{\R^6\times\S^2}B_{ij}
	A_{ij}[h_i^\parallel,h_j^\parallel]^2 M_iM_j^* dvdv^*d\sigma \\
	&= \frac14\sum_{i=1}^n\sum_{j\neq i}\int_{\R^6\times\S^2}B_{ij}
	\big((u_i-u_j)\cdot(v'-v) + (e_i-e_j)(|v'|^2-|v|^2)\big)^2 M_iM_j^*dvdv^*d\sigma.
\end{align*}
Using the symmetry of $B_{ij}$ (thanks to assumption (A1)) and of $M_iM_j^*$ 
with respect to $v$,
the function $G(v,v^*,\sigma)=B_{ij}(u_i-u_j)\cdot(v'-v)(|v'|^2-|v|^2)$
is odd with respect to $(v,v^*,\sigma)$ and thus,
the mixed term of the square in the above integral vanishes. Therefore, we obtain
\begin{align}
  -(f^\parallel,L^b(f^\parallel))_{L_v^2}
	&= \frac14\sum_{i=1}^n\sum_{j\neq i} \int_{\R^6\times\S^2}B_{ij}
	\big(|(u_i-u_j)\cdot(v'-v)|^2 + (e_i-e_j)^2(|v'|^2-|v|^2)^2\big) \nonumber \\
	&\phantom{xx}{}\times M_iM_j^*dvdv^*d\sigma. \label{fLf3}
\end{align}

Now, we claim that 
$$
  \int_{\R^6\times\S^2}B_{ij}((u_i-u_j)\cdot(v'-v))^2 M_iM_j^*dvdv^*d\sigma
	= \frac{|u_i-u_j|^2}{3}\int_{\R^6\times\S^2}B_{ij}|v-v'|^2M_iM_j^*dvdv^*d\sigma.
$$
To prove this identity, we write $u_{i,k}$ and $v_k$ for the $k$th component of
the vectors $u_i$ and $v$, respectively. The transformation $(v_k,v_k^*,\sigma_k)
\mapsto -(v_k,v_k^*,\sigma_k)$ for fixed $k$ leaves $B_{ij}$, $M_i$, and $M_j^*$
unchanged but $v_k'\mapsto -v_k'$ such that
$$
  \int_{\R^6\times\S^2}B_{ij}v_k'v_\ell M_iM_j^*dvdv^*d\sigma = 0
	\quad\mbox{for }\ell\neq k.
$$
Furthermore,
$$
  \int_{\R^6\times\S^2}B_{ij}v_kv_\ell M_iM_j^*dvdv^*d\sigma = 0
	\quad\mbox{for }\ell\neq k,
$$
since the integrand is odd. Therefore,
\begin{align*}
  \int_{\R^6\times\S^2} & B_{ij}((u_i-u_j)\cdot(v'-v))^2 M_iM_j^*dvdv^*d\sigma \\
	&= \sum_{k,\ell=1}^3 (u_{i,k}-u_{j,k})(u_{i,\ell}-u_{j,\ell})
	\int_{\R^6\times\S^2}B_{ij}(v_k'-v_k)(v_\ell'-v_\ell)M_iM_j^*dvdv^*d\sigma \\
	&= \sum_{k=1}^3 (u_{i,k}-u_{j,k})^2\int_{\R^6\times\S^2}B_{ij}
	(v_k-v_k')^2 M_iM_j^*dvdv^*d\sigma.
\end{align*}
In fact, we can see that the integral is independent of $k$, and we infer that
\begin{align*}
  \int_{\R^6\times\S^2} & B_{ij}((u_i-u_j)\cdot(v'-v))^2 M_iM_j^*dvdv^*d\sigma \\
	&= \frac13\sum_{k=1}^3 (u_{i,k}-u_{j,k})^2\int_{\R^6\times\S^2}B_{ij}
	|v-v'|^2 M_iM_j^*dvdv^*d\sigma,
\end{align*}
from which the claim follows.

Hence, \eqref{fLf3} can be estimated as
$$
  -(f^\parallel,L^b(f^\parallel))_{L_v^2}
	\ge D^b\sum_{i,j=1}^3\big(|u_i-u_j|^2+(e_i-e_j)^2\big),
$$
where
\begin{equation}\label{D}
  D^b = \min_{1\le i,j\le n}\int_{\R^6\times\S^2}B_{ij}
	\min\left\{\frac13|v-v'|^2,(|v'|^2-|v|^2)^2\right\}M_iM_j^* dvdv^*d\sigma.
\end{equation}
It remains to show that $D^b>0$. The integrand of \eqref{D} vanishes
if and only if $|v'|=|v|$. However, the set
$$
  X = \{(v,v^*,\sigma)\in\R^3\times\R^3\times\S^2: |v'|=|v|\}
$$
is closed since it is the pre-image of $\{0\}$ of the continuous function 
$F(v,v^*,\sigma)=|v'|^2-|v|^2$, i.e.\ $X=F^{-1}(\{0\})$, recalling that $v'$
depends on $(v,v^*,\sigma)$ through \eqref{vprime}. Since 
$X\neq\R^3\times\R^3\times\S^2$, its complement $X^c$ is open and nonempty
and thus has positive Lebesgue measure. Since the integrand in \eqref{D} is
positive on $X^c$, we infer that $D^b>0$. This finishes the proof.
\end{proof}


\subsection{Estimate for the momentum and energy differences}\label{sec.diff}

The last step is to derive lower bounds for the differences
$\sum_{i,j}(|u_i-u_j|^2+(e_i-e_j)^2)$. First, we recall some moment
identities:
\begin{equation}\label{max1}
  \int_{\R^3}M_idv = \rho_i, \quad \int_{\R^3}M_i v_jv_k dv = \rho_i\delta_{jk},
	\quad \int_{\R^3}M_i|v|^4 dv = 15\rho_i
\end{equation}
for all $1\le i\le n$ and $1\le j,k\le 3$.

\begin{lemma}\label{lem.max2}
Let $f\in L_v^2$ with $f_i^\parallel=M_i^{1/2}(\alpha_i+u_i\cdot v+e_i|v|^2)$ for
$1\le i\le n$. Then
$$
  \int_{\R^3}M_i^{1/2}f_i dv = \rho_i(\alpha_i+3e_i),\
	\int_{\R^3}M_i^{1/2}f_i vdv = \rho_iu_i,\
	\int_{\R^3}M_i^{1/2}f_i|v|^2 dv = \rho_i(3\alpha_i+15e_i).
$$
\end{lemma}

\begin{proof}
Decomposing $f=f^\parallel+f^\perp$, where $f^\parallel=\varPi^m(f)$ and 
$f^\perp=f-\varPi^m(f)$, we infer from $M_i^{1/2}\in\ker(L^m)$ (see \eqref{kerLm})
that $(M_i^{1/2},f_i^\perp)_{L_v^2}=0$ and hence, by \eqref{kerLm} again,
$$
  (M_i^{1/2},f_i^\parallel)_{L_v^2} = \int_{\R^3}M_i(\alpha_i + u_i\cdot v
	+ e_i|v|^2)dv = \rho_i(\alpha_i+3e_i).
$$
The other identities can be shown in a similar way.
\end{proof}

\begin{lemma}\label{lem.diff}
For all $f\in\D$, we have
$$
  \sum_{i,j=1}^n\big(|u_i-u_j|^2 + (e_i-e_j)^2\big)
	\ge \frac{1}{C_k}\big(\|f-\varPi^L(f)\|_\H^2 - 2\|f-\varPi^m(f)\|_\H^2\big),
$$
where $u_i$, $e_i$ are the coefficients of the $i$th component of $\varPi^m(f)$
in \eqref{kerLm}, $\varPi^L$ is the projection on $\ker(L)$, $C_k>0$ is given by
\begin{equation}\label{Ck}
  C_k = 60n\rho_\infty
	\max_{1\le k,\ell\le 5n}\left|\sum_{i=1}^n\int_{\R^3}\psi_k\psi_\ell
	\nu_i dv\right|,
\end{equation}
and $(\psi_k)$ is an arbitrary orthonormal basis of $\ker(L^m)$ in $L_v^2$.
\end{lemma}

\begin{proof}
We again decompose $f=f^\parallel+f^\perp$ with $f^\parallel=\varPi^m(f)$ and
$f^\perp=f-f^\parallel$. Then
\begin{equation}\label{aux1}
  \|f-\varPi^L(f)\|_\H^2 \le 2\big(\|f^\perp\|_\H^2 
	+ \|f^\parallel-\varPi^L(f)\|_\H^2\big).
\end{equation}
We estimate first the difference $g:=f^\parallel-\varPi^L(f)=\varPi^m(f)-\varPi^L(f)
\in\ker(L^m)$ (note that $\ker(L)\subset\ker(L^m)$). 
Let $(\psi_k)$ be an arbitrary orthonormal basis of $\ker(L^m)$ in $L_v^2$. 
Because of  \eqref{kerLm} and $\na_v \nu_i\in L^\infty(\R^3)$, we have
$\psi_k\in\H$. Then, by Young's inequality, we find that
\begin{align*}
  \|g\|_\H^2 &= \sum_{i=1}^n\int_{\R^3}\left|\sum_{k=1}^{5n}(g,\psi_k)_{L_v^2}
	\psi_k\right|^2\nu_i(v)dv
	= \sum_{k,\ell=1}^{5n}(g,\psi_k)_{L_v^2}(g,\psi_\ell)_{L_v^2}
	\sum_{i=1}^n\int_{\R^3}\psi_k\psi_\ell\nu_i(v)dv \\
	&= \sum_{k,\ell=1}^{5n}(g,\psi_k)_{L_v^2}(g,\psi_\ell)_{L_v^2}
	(\psi_k,\psi_\ell)_\H \\
	&\le \frac12\max_{1\le k,\ell\le 5n}|(\psi_k,\psi_\ell)_\H|
	\sum_{k,\ell=1}^{5n}\big((g,\psi_k)_{L_v^2}^2+(g,\psi_\ell)_{L_v^2}^2\big) \\
	&= 5n\max_{1\le k,\ell\le 5n}|(\psi_k,\psi_\ell)_\H|
	\sum_{k=1}^{5n}(g,\psi_k)_{L_v^2}^2
	= 5n\max_{1\le k,\ell\le 5n}|(\psi_k,\psi_\ell)_\H|\,\|g\|_{L_v^2}^2.
\end{align*}
Thus, we infer from \eqref{aux1} that
$$
  \|f-\varPi^L(f)\|_\H^2 \le 2\|f^\perp\|_\H^2 
	+ 10n\max_{1\le k,\ell\le 5n}|(\psi_k,\psi_\ell)_\H|\,
	\|f^\parallel-\varPi^L(f)\|_{L_v^2}^2.
$$
Because of $\ker(L)\subset\ker(L^m)$, we have $\varPi^m\varPi^L=\varPi^L$ and
\begin{align*}
  \|f^\parallel-\varPi^L(f)\|_{L_v^2}^2
	&= \|f^\parallel\|_{L_v^2}^2 - 2(\varPi^m(f),\varPi^L(f))_{L_v^2}
	+ \|\varPi^L(f)\|_{L_v^2}^2 \\
	&= \|f^\parallel\|_{L_v^2}^2 - 2(f,\varPi^L(f))_{L_v^2}
	+ \|\varPi^L(f)\|_{L_v^2}^2 
	= \|f^\parallel\|_{L_v^2}^2 - \|\varPi^L(f)\|_{L_v^2}^2.
\end{align*}
Consequently, setting $k_0=10n\max_{1\le k,\ell\le n}|(\psi_k,\psi_\ell)_\H|$,
\begin{equation}\label{aux2}
  \|f-\varPi^L(f)\|_\H^2 \le 2\|f^\perp\|_\H^2 
	+ k_0\big(\|f^\parallel\|_{L_v^2}^2 - \|\varPi^L(f)\|_{L_v^2}^2\big).
\end{equation}

Next, we compute the $L_v^2$ norms of $f^\parallel$ and $\varPi^L(f)$.
Moment identities \eqref{max1} show that
\begin{align*}
  \|f^\parallel\|_{L_v^2}^2 
	&= \sum_{i=1}^n\int_{\R^3}M_i(\alpha_i + u_i\cdot v + e_i|v|^2)^2 dv \\
	&= \sum_{i=1}^n\int_{\R^3} M_i(\alpha_i^2 + (u_i\cdot v)^2 + e_i^2|v|^4
	+ 2\alpha_i e_i|v|^2)dv \\
	&= \sum_{i=1}^n \rho_{\infty,i}(\alpha_i^2 + |u_i|^2 + 15e_i^2 + 6\alpha_i e_i).
\end{align*}
For the computation of the $L_v^2$ norm of $\varPi^L(f)$, we choose the following 
orthonormal basis $(\phi_j)=(\phi_{j,i})_{i=1,\ldots,n}$ of $\ker(L)$ in $L_v^2$:
$$
  \phi_{j,i} = \rho_{\infty}^{-1/2}M_j^{1/2}\delta_{ij}, \quad
	\phi_{n+k,i} = \rho_{\infty}^{-1/2}M_i^{1/2}v_k, \quad
	\phi_{n+4,i} = (6\rho_\infty)^{-1/2}M_i^{1/2}(|v|^2-3),
$$
where $1\le j\le n$ and $1\le k\le 3$. Then, using the moment identities
of Lemma \ref{lem.max2},
$$
  \|\varPi^L(f)\|_{L_v^2}^2
	= \sum_{j=1}^{n+4}(f,\phi_{j})_{L_v^2}^2 
	= \sum_{i=1}^n\rho_{\infty,i}(\alpha_i+3e_i)^2 
	+ \rho_\infty\left|\sum_{i=1}^n\frac{\rho_{\infty,i}}{\rho_\infty}u_i\right|^2
	+ 6\rho_\infty\left(\sum_{i=1}^n\frac{\rho_{\infty,i}}{\rho_\infty}e_i\right)^2.
$$

Inserting the above identities for $\|f^\parallel\|_{L_v^2}^2$ and
$\|\varPi^L(f)\|_{L_v^2}^2$ into \eqref{aux2}, we conclude that
\begin{align*}
  \|f-\varPi^L(f)\|_\H^2
	&\le 2\|f-\varPi^m(f)\|_\H^2 
	+ k_0\rho_\infty\left(\sum_{i=1}^n\frac{\rho_{\infty,i}}{\rho_\infty}|u_i|^2 
	- \left|\sum_{i=1}^n\frac{\rho_{\infty,i}}{\rho_\infty}
	u_i\right|^2\right) \\
	&\phantom{xx}{}+ 6k_0\rho_\infty\left(\sum_{i=1}^n
	\frac{\rho_{\infty,i}}{\rho_\infty} e_i^2
	- \left(\sum_{i=1}^n\frac{\rho_{\infty,i}}{\rho_\infty}e_i\right)^2\right).
\end{align*}
Then, if the inequalities
\begin{align}
  \sum_{i=1}^n \frac{\rho_{\infty,i}}{\rho_\infty}|u_i|^2 
	- \left|\sum_{i=1}^n\frac{\rho_{\infty,i}}{\rho_\infty}u_i\right|^2
	&\le \sum_{i,j=1}^n|u_i-u_j|^2, \label{aux.ineq1} \\
  \sum_{i=1}^n\frac{\rho_{\infty,i}}{\rho_\infty} e_i^2 
	- \left(\sum_{i=1}^n\frac{\rho_{\infty,i}}{\rho_\infty}e_i\right)^2
	&\le \sum_{i,j=1}^n(e_i-e_j)^2 \label{aux.ineq2}
\end{align}
hold, the lemma follows with $C_k=6k_0\rho_\infty$.

It remains to prove \eqref{aux.ineq1} and \eqref{aux.ineq2}. To this end, we define
the following scalar product on $\R^{3n}$:
$$
  (u,v)_\rho = \sum_{i=1}^n\frac{\rho_{\infty,i}}{\rho_\infty}u_i\cdot v_i, \quad
	u=(u_1,\ldots,u_n),\,v=(v_1,\ldots,v_n)\in\R^{3n},
$$
where $u_i\cdot v_i$ denotes the usual scalar product in $\R^3$. The
corresponding norm is $\|u\|_\rho = (u,u)_\rho^{1/2}$. Then
${\mathbf 1}=(1,\ldots,1)\in\R^{3n}$ satisfies $\|{\mathbf 1}\|_\rho=1$.
The elementary identity
$$
  \|u\|_\rho^2 - (u,{\mathbf 1})_\rho^2 = \|u-(u,{\mathbf 1})_\rho{\mathbf 1}\|_\rho^2
$$
can be equivalently written as
$$
  I := \sum_{i=1}^n \frac{\rho_{\infty,i}}{\rho_\infty}|u_i|^2 
	- \left|\sum_{i=1}^n\frac{\rho_{\infty,i}}{\rho_\infty}u_i\right|^2
	= \sum_{i=1}^n\frac{\rho_{\infty,i}}{\rho_\infty}\left|u_i 
	- \sum_{j=1}^n\frac{\rho_{\infty,j}}{\rho_\infty}u_j
	\right|^2.
$$
Then, using $\sum_{j=1}^n\rho_{\infty,j}=\rho_\infty$,
\begin{align*}
  I &= \sum_{i=1}^n\frac{\rho_{\infty,i}}{\rho_\infty}
	\left|\left(1-\frac{\rho_{\infty,i}}{\rho_\infty}\right)u_i
	- \sum_{j\neq i}\frac{\rho_{\infty,j}}{\rho_\infty}u_j\right|^2
	= \sum_{i=1}^n\frac{\rho_{\infty,i}}{\rho_\infty}
	\left|\sum_{j\neq i}\frac{\rho_{\infty,j}}{\rho_\infty}
	(u_i-u_j)\right|^2 \\
	&= \sum_{i=1}^n\frac{\rho_{\infty,i}}{\rho_\infty}
	\left(\sum_{k\neq i}\frac{\rho_{\infty,k}}{\rho_\infty}\right)^2
  \left|\frac{\sum_{j\neq i}(\rho_{\infty,j}/\rho_\infty)
	(u_i-u_j)}{\sum_{k\neq i}\rho_{\infty,k}/\rho_\infty}  \right|^2 \\
  &= \sum_{i=1}^n\frac{\rho_{\infty,i}}{\rho_\infty}
	\left(\sum_{k\neq i}\frac{\rho_{\infty,k}}{\rho_\infty}\right)^2
	\left|\sum_{j\neq i}\lambda_j(u_i-u_j)\right|^2,
\end{align*}
where $\lambda_j=(\rho_{\infty,j}/\rho_\infty)(\sum_{k\neq i}
(\rho_{\infty,k}/\rho_\infty))^{-1}$. 
Since $\sum_{j\neq i}\lambda_j=1$, we may apply Jensen's inequality to this convex
combination, leading to
\begin{align*}
  I &\le \sum_{i=1}^n\frac{\rho_{\infty,i}}{\rho_\infty}
	\left(\sum_{k\neq i}\frac{\rho_{\infty,k}}{\rho_\infty}\right)^2
	\sum_{j\neq i}\lambda_j|u_i-u_j|^2 \\
  &= \sum_{i=1}^n\frac{\rho_{\infty,i}}{\rho_\infty}
	\left(1-\frac{\rho_{\infty,i}}{\rho_\infty}\right)
	\sum_{j\neq i}\frac{\rho_{\infty,j}}{\rho_\infty}|u_i-u_j|^2
	\le \sum_{i,j=1}^n|u_i-u_j|^2,
\end{align*}
since $\rho_{\infty,j}\le\rho_\infty$. This ends the proof.
\end{proof}

Now, we are able to prove Theorem \ref{thm.spec}.

\begin{proof}[Proof of Theorem \ref{thm.spec}]
By Lemmas \ref{lem.ortho}, \ref{lem.Lb}, and \ref{lem.diff}, we obtain
\begin{align*}
  -(f,L(f))_{L_v^2} 
	&\ge (C^m-4\eta)\|f-f^\parallel\|_\H^2
	+ \frac{\eta D^b}{8}\sum_{i,j=1}^n\big(|u_i-u_j|^2+(e_i-e_j)^2\big) \\
	&\ge \left(C^m-4\eta-\frac{\eta D^b}{4C_k}\right)\|f-f^\parallel\|_\H^2 
	+ \frac{\eta D^b}{8C_k}\|f-\varPi^L(f)\|_\H^2.
\end{align*}
The first term on the right-hand side is nonnegative if we choose
$\eta = \min\{1,4C^mC_k/(16C_k+D^b)\}$, and 
estimate \eqref{spec.gap} follows with $\lambda=\eta D^b/(8C_k)$.
\end{proof}


\section{Convergence to equilbrium}\label{sec.conv}

In this section, we prove Theorem \ref{thm.conv}. The idea of the proof
is to adapt the hypocoercivity method of \cite{MoNe06} to the multi-species
setting. To this end, we need to verify the structural assumptions
(H1)-(H3) in \cite[Theorem 1.1]{MoNe06}. The setting is as follows.

Let $L$ be a closed, densely defined, and self-adjoint operator on 
$\mbox{Dom}(L)\subset L_v^2$ such that
$L=K-\Lambda$ and the operators $K$ and $\Lambda$ satisfy the 
following assumptions:

\renewcommand{\labelenumi}{(H\theenumi)}
\begin{enumerate}
\item The operator $\Lambda$ is coercive in the following sense:
There exist a norm $\|\cdot\|_\H$ on $\H\subset L_v^2$
and positive constants $\bar\nu_i$ ($0\le i\le 4$) such that for all
$f\in\mbox{Dom}(L)\subset\H$,
\begin{align}
  & \bar\nu_0\|f\|_{L_v^2}^2 \le \bar\nu_1\|f\|_\H^2 
	\le (f,\Lambda(f))_{L_v^2}\le \bar\nu_2\|f\|_\H^2, \label{H1.1} \\
	& (\na_v f,\na_v\Lambda(f))_{L_v^2} \ge \bar\nu_3\|\na_v f\|_\H^2 
	- \bar\nu_4\|f\|_{L_v^2}^2. \label{H1.2}
\end{align}
Moreover, there exists a constant $C_L>0$ such that for all $f$, $g\in\mbox{Dom}(L)$,
\begin{equation}\label{H1.3}
  (L(f),g)_{L_v^2} \le C_L\|f\|_\H\|g\|_\H.
\end{equation}
\item The operator $K$ has a regularizing effect in the following sense:
For all $\eps>0$, there exists $C(\eps)>0$ such that for all
$f\in H_v^1$,
$$
  (\na_v f,\na_v K(f))_{L_v^2} \le \eps\|\na_v f\|_{L_v^2}^2
	+ C(\eps)\|f\|_{L_v^2}^2.
$$
\item The operator $L$ has a finite-dimensional
kernel and the following local spectral-gap assumption holds: There exists
$\lambda>0$ such that for all $f\in\mbox{Dom}(L)$,
$$
  -(f,L(f))_{L_v^2} \ge\lambda\|f-\varPi^L(f)\|_\H^2,
$$
where $\varPi^L$ is the projection on $\ker(L)$.
\end{enumerate}

Assumption (H3) is a consequence of Theorem \ref{thm.spec}.
Next, we verify assumption (H1). Using Lemma \ref{lem.Lambda} and the 
continuous embedding $\H\hookrightarrow L_v^2$, we see that \eqref{H1.1} holds.
For the proof of \eqref{H1.2}, we employ Young's inequality:
\begin{align*}  
  (\na_v f,\na_v\Lambda(f))_{L_v^2}
	&= \sum_{i=1}^n\int_{\R^3}\na_v f_i\cdot\na_v(\nu_i f_i)dv \\
	&= \sum_{i=1}^n\left(\int_{\R^3}f_i\na_v f_i\cdot\na_v\nu_i dv
	+ \int_{\R^3}|\na_v f_i|^2 \nu_i dv\right) \\
	&\ge \frac12\sum_{i=1}^n\left(-\int_{\R^3}\frac{|\na_v\nu_i|^2}{\nu_i}f_i^2 dv
	+ \int_{\R^3}|\na_v f_i|^2\nu_i dv\right) \\
	&\ge \bar\nu_3\|\na_v f\|_\H^2 - \bar\nu_4\|f\|_{L_v^2}^2,
\end{align*}
where $\bar\nu_3=1/2$ and $\bar\nu_4=\max_{1\le i\le n}\sup_{v\in\R^3}
|\na_v\nu_i|^2/(2\nu_i)$. Note that $\bar\nu_4$ is finite since $\na_v\nu_i$
is bounded and $\nu_i$ is strictly positive (see Lemma \ref{lem.nu}).
Finally, inequality \eqref{H1.3} follows from the decomposition $L=K-\Lambda$,
the compactness and hence continuity of $K$, the explicit expression for $\Lambda$,
and the Cauchy-Schwarz inequality applied to $(L(f),g)_{L_v^2}$.

It remains to verify assumption (H2). Let $N:=\rho_{\infty,i}^{-1/2}M_i^{1/2}
= (2\pi)^{-3/4}\exp(-|v|^2/4)$. We decompose $K=K^{(1)}-K^{(2)}$, where
$K^{(j)}=(K^{(j)}_1,\ldots,K^{(j)}_n)$ and
\begin{align*}
  K_i^{(1)} &= \sum_{j=1}^n\int_{\R^3\times\S^2}B_{ij}M_i^{1/2}M_j^*
	\left(\frac{f_i'}{(M_i')^{1/2}} + \frac{f_j'^*}{(M_j'^*)^{1/2}}\right)dv^*d\sigma, \\
	K_i^{(2)} &= \sum_{j=1}^n\int_{\R^3\times\S^2}B_{ij}(M_iM_j^*)^{1/2}f_j^* dv^*d\sigma
\end{align*}
for $1\le i\le n$.
Because of $M_k'M_k'^*=M_kM_k^*$ for all $k$, we find that
\begin{align*}
  K_i^{(1)}(f) &= \sum_{j=1}^n\int_{\R^3\times\S^2}B_{ij}M_i^{1/2}M_j^*
	\left(\frac{(M_i'^*)^{1/2}f_i'}{(M_i'M_i'^*)^{1/2}}
	+ \frac{(M_j')^{1/2}f_j'^*}{(M_j'^*M_j')^{1/2}}\right)dv^*d\sigma \\
	&= \sum_{j=1}^n\int_{\R^3\times\S^2}B_{ij}M_i^{1/2}M_j'
	\left(\frac{(M_i'^*)^{1/2}f_i'}{(M_iM_i^*)^{1/2}}
	+ \frac{(M_j')^{1/2}f_j'^*}{(M_j^*M_j)^{1/2}}\right)dv^*d\sigma \\
  &= \sum_{j=1}^n\int_{\R^3\times\S^2}B_{ij}
	\big(\rho_{\infty,j}^{1/2}N'^*f_i' + \rho_{\infty,i}^{1/2}N'f_j'^*\big) 
	\rho_{\infty,j}^{1/2}N^* dv^*d\sigma.
\end{align*}
The transformation $\sigma\mapsto-\sigma$ leaves $v$ and $v^*$ unchanged
and exchanges $v'$ and $v'^*$. Assumption (A5) ($b_{ij}$ is an even function)
ensures that $B_{ij}$ is unchanged under this transformation. Therefore,
$$
  \int_{\R^3\times\S^2} B_{ij}f_i' N'^* N^* dv^*d\sigma
	= \int_{\R^3\times\S^2}B_{ij}f_i'^*N'N^* dv^*d\sigma,
$$
and we can write $K^{(1)}$ as
\begin{equation}\label{aux4}
  K_i^{(1)}(f) = \frac12\sum_{j=1}^n\rho_{\infty,j}^{1/2}
	\big(\rho_{\infty,j}^{1/2}K_{ij}^{(1)}(f_i)
	+ \rho_{\infty,i}^{1/2}K_{ij}^{(1)}(f_j)\big),
\end{equation}
where
$$
  K_{ij}^{(1)}(f_k) = \int_{\R^3\times\S^2}B_{ij}(N'^*f_k'+N'f_k'^*)N^* dv^*d\sigma,
	\quad 1\le i,j,k\le n.
$$
Note that $K_{ij}^{(1)}=K_{ji}^{(1)}$.
In a similar way, we can decompose the operator $K^{(2)}$:
$$
  K_i^{(2)}(f) = \sum_{j=1}^n(\rho_{\infty,i}\rho_{\infty,j})^{1/2}N
	\int_{\R^3\times\S^2}B_{ij}f_j^* N^* dv^*d\sigma
	= \sum_{j=1}^n(\rho_{\infty,i}\rho_{\infty,j})^{1/2}K_{ij}^{(2)}(f_j),
$$
where
$$
  K_{ij}^{(2)}(f_j) = N\int_{\R^3\times\S^2}B_{ij}N^*f_j^*dv^*d\sigma.
$$

Next, we estimate the derivatives of $K_{ij}^{(\ell)}$. It is shown in
\cite[Eqs.~(5.15)-(5.18)]{MoNe06} that for all $\eps>0$, there exists
$C(\eps)>0$ such that for any $f\in H_v^1$, $1\le i,j,k\le n$, and $\ell=1,2$,
\begin{equation}\label{aux5}
  \|\na_v K_{ij}^{(\ell)}(f_k)\|_{L_v^2}^2
	\le \eps\|\na_v f_k\|_{L_v^2}^2 + C(\eps)\|f_k\|_{L_v^2}^2.
\end{equation}
Then we infer from \eqref{aux4} that
\begin{align*}
  \|\na_v K^{(1)}(f)\|_{L_v^2}^2
	&= \sum_{i=1}^n\bigg\|\frac12\sum_{j=1}^n
	\rho_{\infty,j}^{1/2}\big(\rho_{\infty,j}^{1/2}\na_v K_{ij}^{(1)}(f_i)
	+ \rho_{\infty,i}^{1/2}\na_v K_{ij}^{(1)}(f_j)\big)\bigg\|_{L_v^2}^2 \\
	&\le \frac{n}{4}\sum_{i,j=1}^n\left\|\rho_{\infty,j}^{1/2}
	\big(\rho_{\infty,j}^{1/2}\na_v K_{ij}^{(1)}(f_i)
	+ \rho_{\infty,i}^{1/2}\na_v K_{ij}^{(1)}(f_j)\big)\right\|_{L_v^2}^2 \\
	&\le n(\max_{1\le i\le n}\rho_{\infty,i})^2\sum_{i,j=1}^n
	\|\na_v K_{ij}^{(1)}(f_i)\|_{L_v^2}^2.
\end{align*}
Thus, by \eqref{aux5}, it follows that for $\ell=1$,
$$
  \|\na_v K^{(\ell)}(f)\|_{L_v^2}^2 \le n^2(\max_{1\le i\le n}\rho_{\infty,i})^2
	\sum_{i=1}^n\big(\eps\|\na_v f_i\|_{L_v^2}^2 + C(\eps)\|f_i\|_{L_v^2}^2\big).
$$
A similar computation shows that this estimate also holds for $\ell=2$.
We infer that
$$
  \|\na K(f)\|_{L_v^2}^2 \le 4n^2(\max_{1\le i\le n}\rho_{\infty,i})^2\sum_{i=1}^n
	\big(\eps\|\na_v f_i\|_{L_v^2}^2 + C(\eps)\|f_i\|_{L_v^2}^2\big).
$$
This proves assumption (H2) since $\eps>0$ is arbitrary.

\begin{proof}[Proof of Theorem \ref{thm.conv}]
We have verified that assumptions (H1)-(H3) are satisfied.
Then, using exactly the same arguments as in the proof of Theorem 1.1 in
\cite{MoNe06}, but now for the multi-species case, we conclude the
exponential decay \eqref{decay1} of the semigroup $\e^{tB}$, which is
the first property of the theorem.

It remains to show that the decay estimate \eqref{decay2} follows from
\eqref{decay1}. For this, we write the initial value $f_I$ as
$f_I=\varPi^B(f_I) + (I-\varPi^B)(f_I)$, where $\varPi^B$ is the projection
onto $\ker(B)$ in $L_{x,v}^2$. Then the solution to \eqref{LBE} is given by
$$
  f(t) = \e^{tB}f_I = \e^{tB}\varPi^B(f_I) + \e^{tB}(I-\varPi^B)(f_I), \quad
	t\ge 0.
$$
We have already shown that
$$
  \|\e^{tB}(I-\varPi^B)g\|_{H_{x,v}^1} \le C\e^{-\tau t}\|g\|_{H_{x,v}^1}
	\quad\mbox{for all }g\in H_{x,v}^1,\ t>0.
$$
In particular, the choice $g=(I-\varPi^B)(f_I)$ and the property
$(I-\varPi^B)^2=I-\varPi^B$ lead to
$$
  \|\e^{tB}(I-\varPi^B)(f_I)\|_{H_{x,v}^1} 
	\le C\e^{-\tau t}\|(I-\varPi^B)(f_I)\|_{H_{x,v}^1}.
$$

It remains to prove that $f_\infty = \varPi^B(f_I) = \e^{tB}\varPi^B(f_I)$
is the global equilibrium. Since $B\varPi^B(f_I)=0$ and $\varPi^B(f_I)$
does not depend on time, the constant-in-time function $g=\varPi^B(f_I)$
is the unique solution to the Cauchy problem
$$
  \pa_t g = Bg, \quad t>0, \quad g(0) = \varPi^B(f_I).
$$
This shows that $\varPi^B(f_I) = \e^{tB}g(0) = \e^{tB}\varPi^B(f_I)$
and finishes the proof.
\end{proof}


\end{document}